\documentclass[a4paper,12pt,reqno]{amsart}

\usepackage{amsthm}
\usepackage{amsmath}
\usepackage{amssymb}
\usepackage{latexsym}
\usepackage{enumerate}
\usepackage{cite}
\usepackage{color}

\newtheorem{theorem}{Theorem}[section]
\newtheorem{rem}[theorem]{Remark}

\newtheorem{prop}[theorem]{Proposition}
\newtheorem{lem}[theorem]{Lemma} 
\makeatletter
\@addtoreset{equation}{section}

\makeatother
\newcommand{\N}{\mathbb{N}}
\newcommand{\Z}{\mathbb{N}_0}
\newcommand{\R}{\mathbb{R}}
    
\newcommand{\Li}{L^\infty}
\newcommand{\Sc}{\mathcal{S}}         
\newcommand{\Scp}{\mathcal{S}^\prime} 
\newcommand{\F}{\mathcal{F}}         
   
\newcommand{\tp}{t^\prime}
\newcommand{\Pa}{\partial}
\newcommand{\Pt}{\partial_t}

\newcommand{\bra}[1]{ \left<  #1 \right>  }
\newcommand{\braa}[1]{\left(  #1 \right)  }
\newcommand{\brab}[1]{\left\{ #1 \right\} }
\newcommand{\brac}[1]{\left[  #1 \right]  }
\newcommand{\abs}[1]{ \left|  #1 \right|  }
\newcommand{\nr}[1]{  \left\| #1 \right\| }

\newcommand{\n}{\mathcal{N}}
\newcommand{\s}{{\mbox{\boldmath $S$}}}
\newcommand{\aY}{\acute{Y}}
\newcommand{\wZ}{\widetilde{Z}}
\newcommand{\Dl}{D_\lambda}

\newcommand{\Gs}{\Gamma}
\newcommand{\mh}{\mathcal{H}}
\newcommand{\oi}{\omega^{-1}}
\newcommand{\vD}{\varDelta_z}
\newcommand{\li}{L^\infty H^1}
\newcommand{\wW}{\widetilde{W}}

\makeatletter 
\newcommand{\bbb}[1]{\bBigg@{#1}} 
\makeatother

\newcommand{\Miijj}[4]{
\begin{pmatrix}
   #1 & #2 \\
   #3 & #4
\end{pmatrix}}

\begin{document}
\title[%
Inverse scattering problems 
for NLKG]{%
On inverse scattering  
for the two-dimensional nonlinear Klein-Gordon equation}

\author[H. Sasaki]{%
Hironobu Sasaki$^*$\\
Department of Mathematics and Informatics, 
Chiba University, 263--8522, Japan.}

\thanks{$^*$Supported by 
JSPS KAKENHI Grant Number 22K03367.}

\thanks{$^*$email:
\texttt{sasaki@math.s.chiba-u.ac.jp}} 


\maketitle

\begin{abstract}
The inverse scattering problem 
for the two-dimensional nonlinear Klein-Gordon equation 
$u_{tt}-\Delta u + u = \mathcal{N}(u)$
is studied. 
We assume that 
the unknown nonlinearity $\mathcal{N}$ 
of the equation 
satisfies 
$\mathcal{N}\in C^\infty(\mathbb{R};\mathbb{R})$, 
$\mathcal{N}^{(k)}(y)=O(|y|^{\max\{ 3-k,0 \}})$ 
($y \to 0$)  
and 
$\mathcal{N}^{(k)}(y)=O(e^{c y^2})$ 
($|y| \to \infty$)
for any $k=0,1,2,\cdots$.
Here,
$c$ is a positive constant.
We establish a reconstraction formula of 
$\mathcal{N}^{(k)}(0)$
($k=3,4,5,\cdots$) 
by the knowledge of the scattering operator 
for the equation.
As an application,
we also give an expression for 
higher order G\^{a}teaux differentials
of the scattering operator at 0.
\end{abstract}

\section{Introduction}
We consider 
the inverse scattering problem 
for the nonlinear Klein-Gordon equation 
\begin{align}\label{NLKG}
u_{tt}-\Delta u + u = \n(u)
\end{align} 
in the two space dimensions. 
Here, 
$u=u(t,x)$ is a $\R$-valued unknown function 
defined on $\R \times \R^2$,
$\Delta=\Pa_1^2+\Pa_2^2$ is the Laplacian in $x$,
$\Pa_j=\Pa/\Pa x_j$ 
($j=1,2$) 
and $\n(u)$ is the nonlinearity. 
More precisely,
throughout this paper we assume 
$\n\in C^\infty(\R;\R)$, 
\begin{align}\label{N1}
\n^{(k)}(y)
=
O\braa{
|y|^{\max\{ 3-k,0 \}}
}
\quad
(y\to 0),
\quad
k\in \Z
\end{align} 
and 
\begin{align}\label{N2}
\n^{(k)}(y)
=
O\braa{
e^{c y^2}
}
\quad
(|y|\to \infty),
\quad
k\in \Z,
\end{align} 
where
$\Z=\N \cup \{ 0 \}$
and 
$c$ is a positive constant.
The following functions 
$\n_P$ and $\n_{c_1,c_2}$ are 
examples of $\n$:
\begin{enumerate}[(Ex.1)]
  \item 
For polynomial $P(y)$, 
we put $\n_P(y)=P(y)y^3$.
  \item 
For constants $c_1,c_2 \in \R$,
we put $\n_{c_1,c_2}(y)=c_1(e^{c_2 y^2}-1)y$.
\end{enumerate}
Our aim of this paper 
is to identify the unknown nonlinearity $\n$ 
by the knowledge of the scattering operator 
$\s$ for (\ref{NLKG}).

\subsection{Direct scattering}
Before we consider inverse scattering problems,
we first mention direct problems.
As for the existence of $\s$, 
Nakamura and Ozawa 
\cite{NakamuraOzawa2001} 
first studied 
the Klein-Gordon equation 
with exponential-type nonlinearity, 
and suggests that
if $\n$ satisfies (\ref{N1}) and (\ref{N2}) 
then $\s$ is defined on a neighborhood 
in the energy space 
$\mh:=H^1(\R^2) \oplus L^2(\R^2)$ of 0. 
In \cite{NakamuraOzawa2001},
it was also indicated that 
$\mh$ seems to be the critical space for (\ref{NLKG}) 
if $\n(u)$ behaves like $e^{c |u|^2}$  
as $|u|\to\infty$.
If $\n=\n_{c_1,c_2}$ with $c_1>0$, 
then 
$\mh$ is actually the critical space 
(see, e.g., 
\cite{IbrahimMajdoubMasmoudi2006,
IbrahimMajdoubMasmoudiNakanishi2009,
IbrahimMasmoudiNakanishi2009}).
As for 
dispersive equation with exponential-type nonlinearity,
see also 
\cite{IkedaInuiOkamoto2020,
NakamuraOzawa1998,
NakamuraOzawa1999,
Sasaki2024}.

In order to state 
the existence theorem for $\s$ precisely,
we introduce some notations and propositions.
For $p\in [1,\infty]$ and $s\in \R$, 
we denote 
the Lebesgue space $L^p(\R^2)$ by $L^p$,
its norm by $\nr{\cdot}_p$, 
the Sobolev space $H^s(\R^2)$ by $H^s$,
and the inhomogeneous Besov space $B^s_{p,2}(\R^2)$ by $B^s_p$.
For the definition of $B^s_p$, 
see, e.g., \cite{BL}.
Put 
$\bra{x}=\sqrt{1+|x|^2}$ ($x\in \R^2$) 
and 
$\omega=\sqrt{1-\Delta}$.
For $a,b\in \R$,
we set 
$a\vee   b=\max\{ a,b \}$,
$a\wedge b=\min\{ a,b \}$ 
and $a_+=a \vee 0$.
For $s\in \R$,
the $H^s$-norm is given 
by $\nr{f}_{H^s}=\nr{\omega^s f}_2$. 
Let $\nr{\cdot}$ be 
the norm of $\mh$ defined by 
\begin{align*}
\nr{\phi}
=
\nr{f}_{H^1}\vee\nr{g}_2,
\quad
\phi=\binom{f}{g} \in \mh.
\end{align*} 
By $\Sc$ and $\Scp$, 
we denote the Schwartz space $\Sc(\R^2)$ 
and the set of all tempered distributions on $\R^2$, 
respectively. 
For a Banach space $\mathcal{X}$, 
$p\in [1,\infty]$ and $k\in \Z$, 
we define $L^p \mathcal{X} =L^p(\R;\mathcal{X})$, 
and $C^k \mathcal{X}$ denotes the set of all functions  
which are defined on $\R$ 
and $k$-th order continuously differentiable 
in the $\mathcal{X}$-sense. 
Moreover, 
we put the closed ball 
\begin{align*}
B_\eta \mathcal{X} =\brab{
f\in \mathcal{X}; \nr{f}_{\mathcal{X}} \le \eta
},
\quad
\eta>0
\end{align*}  
and the function space 
\begin{align*}
C_b^k \mathcal{X} = \brab{
f\in C^k \mathcal{X}; 
\frac{\Pa^\ell}{\Pa t^\ell} f \in \Li \mathcal{X}, 
\ \ell=0,\cdots,k
}.
\end{align*} 
By $\Pt$, 
we denote the time derivative operator in the $L^2$-sense.
For $v\in C^1 L^2$ and $t\in \R$,
we set $(1,\Pt)v(t)=(v(t),\Pt v(t))$.
Define 
$\aY= L^{4/3}B^{1/2}_{4/3}$ 
and 
$Z=C_b^0 H^1 \cap C_b^1 L^2 \cap L^4B^{1/2}_4$.
Furthermore, 
we give 
\begin{align*}
\nr{v}_Z
=
     \nr{    v}_{\li} 
\vee \nr{    v}_{L^4B^{1/2}_4} 
\vee \nr{\Pt v}_{\Li L^2},
\quad
v\in Z. 
\end{align*} 
Put 
\begin{align*}
w_\phi(t)
=\cos(t\omega) f +\frac{\sin(t\omega)}{\omega}g, 
\quad \phi=\binom{f}{g} 
\in \Scp\oplus\Scp, \ t\in \R.
\end{align*}
Then $w_\phi$ is a mild solution 
to the free Klein-Gordon equation 
$w_{tt}-\Delta w + w = 0$ 
with initial conditions $w(0,x)=f(x)$ and $w_t(0,x)=g(x)$.
For $G\in \Scp(\R^{1+2}_{(t,x)})$,
we define the function $\Gs G$ by 
\begin{align*}
(\Gs G)(t,x)
=
\int_{-\infty}^t 
\frac{\sin(t-\tp)\omega}{\omega}
G(\tp,x)d\tp
\end{align*} 
if the right hand side makes sense.
Then we formally see that  
$v:=\Gs G$ solves the equation 
$v_{tt}-\Delta v + v = G$. 
We remark that 
$\Gs$ becomes a bounded linear operator 
from $\aY$ to $Z$. 
For detail, 
see Proposition \ref{prop:St} below.
For $a,b \ge 0$ and 
for parameters 
$p_1,\cdots,p_m$,
we write $a\lesssim b$ 
(resp. $a\lesssim_{p_1,\cdots,p_m} b$)
if we have 
the inequality $a \le C b$ 
with a positive constant $C$ 
dependent at most on $\n$
(resp. $\n,p_1,\cdots,p_m$).

The existence of the scattering operator 
for (\ref{NLKG}) is given by 
\cite{NakamuraOzawa2001}. 
The detailed statement is as follows:

\begin{prop}\label{prop:S}
There exists some $\eta_0\in (0,1]$ such that 
for any $\phi_-\in B_{\eta_0}\mh$ 
there exist a unique time-global solution $u \in Z$
to the integral equation of the form 
\begin{align*}
u(t)=w_{\phi_-}(t)+\Gs\n(u), 
\quad t\in \R ,
\end{align*}
and a unique function $\phi_+ \in \mh$, 
which satisfies 
\begin{align*}
\phi_+
=
\phi_-
+
\int_{\R}
\binom{
-\oi \sin(t\omega)
}{
 \cos(t\omega)    
}
\n(u(t))
dt,
\end{align*}
such that 
$\nr{\n(u)}_{\aY} \lesssim \nr{\phi_-}^3$,
\begin{align*}
\lim_{t\to -\infty}
 \nr{
 (1,\Pt) u(t)
 -
 (1,\Pt) w_{\phi_-}(t)
 }
=0
\end{align*} 
and 
\begin{align*}
\lim_{t\to +\infty}
 \nr{
 (1,\Pt) u(t)
 -
 (1,\Pt) w_{\phi_+}(t)
 } 
=0.
\end{align*} 
Hence the scattering operator 
\begin{align*}
\s: B_{\eta_0}\mh 
\ni \phi_-\mapsto \phi_+
\in \mh
\end{align*} 
for (\ref{NLKG}) is well-defined.
\end{prop}

Functions $\phi_-$ and $\phi_+:=\s \phi_-$ 
in the above proposition 
are called input data and output data,
respectively.

\subsection{Inverse scattering and main results}
The inverse scattering problem (ISP)
for nonlinear dispersive equations is 
to identify unknown nonlinearities  
by the knowledge of scattering states 
(see, e.g., %
\cite{%
BarretoUhlmannWang2022,
ChenMurphy-pre,
ChenMurphy-pre-2,
HoganMurphyGrow2023,
HuKillipVisan-pre,
KillipMurphyVisan2023,
MSSSS2018,
MorawetzStrauss1973,
Sasaki2007,
Sasaki2008,
Sasaki2012,
Sasaki2024-2,
SasakiSuzuki2011,
SasakiWatanabe2005,
Strauss1974,
Watanabe2002,
Weder1997,
Weder2000,
Weder2000-2,
Weder2001,
Weder2001-2,
Weder2002,
Weder2005-1,
Weder2005-2}
and references therein).

In the present paper, 
we uniquely determine 
the exact value of $\n^{(N)}(0)$ 
($N=3,4,\cdots$).
Remark that 
by Assumption (\ref{N1}) 
we already obtain $\n^{(N)}(0)=0$ 
if $N=0,1,2$.
In the case of $N=3,4$, 
the following ``expansion" of $\s$ seems to be useful:
\begin{align}\label{SAL}
\bra{
\s(\lambda\phi) - \lambda\phi,
\Miijj{0}{-\omega^{-2}}{1}{0}\phi
}_{\mh}
=
\int_{\R^{1+2}}
\n(\lambda w_\phi)w_\phi
d(t,x)
+
O(\lambda^5)
\quad
(\lambda\to +0),
\end{align} 
where 
$\phi \in \mh$,
$w_\phi$ is the free solution mentioned above 
and 
$\bra{\cdot,\cdot}_{\mh}$
is the inner product on $\mh$ 
given by 
\begin{align*}
\bra{\phi_1,\phi_2}_{\mh}
=
\bra{\omega f_1,\omega f_2}_{L^2}
+
\bra{g_1,g_2}_{L^2},
\quad
\phi_j=\binom{f_j}{g_j} \in \mh, \ j=1,2.
\end{align*} 
Indeed, 
we formally obtain 
\begin{align}\label{f:3-4}
\n^{(N)}(0)
=
\frac{
\displaystyle{
\lim_{\lambda\to +0}
\frac{d^N}{d \lambda^N}
\bra{
\s(\lambda\phi) - \lambda\phi,
\Miijj{0}{-\omega^{-2}}{1}{0}\phi
}_{\mh}
}
}{
\displaystyle{
\int_{\R^{1+2}}
w_\phi(t,x)^{N+1}
d(t,x)
}
}
\end{align}
if $N=3,4$ 
and 
the denominator is not zero.
Formula (\ref{SAL}) is 
a kind of 
the small amplitude limit.
In the study of 
inverse scattering problems 
for nonlinear dispersive equations,
the small amplitude limit is the most essential tool 
(see, e.g., 
\cite{%
MorawetzStrauss1973,
Sasaki2007,
Sasaki2012,
SasakiWatanabe2005,
Strauss1974,
Weder2000,
Weder2002}).

In the case of $N\ge 5$, 
unfortunately,
we can not expect that 
(\ref{f:3-4}) holds true 
because of the term $O(\lambda^5)$ in (\ref{SAL}).
In the present paper,
by applying the method of \cite{Sasaki2024-2}
we establish 
a higher-order expansion of $\s$
whose remainder term is $O(\lambda^{N+1})$, 
and show that 
the exact value of $\n^{(N)}(0)$ 
is uniquely determined. 
In \cite{Sasaki2024-2},  
we consider inverse scattering problems 
for the two-dimensional nonlinear Schr\"{o}dinger equation 
$iu_t+\Delta u= \mathcal{M}(u)$, 
where $\mathcal{M}$ is an unknown complex function 
satisfying some assumptions 
similar to (\ref{N1}) and (\ref{N2}), 
and we determine 
higher order Wirtinger derivatives 
of $\mathcal{M}$ at the origin.
\\

To state main results precisely,
we list some notations.
For a function $h:(0,\infty)\to \R$ 
and for $\lambda>0$
we put 
\begin{align*}
\varDelta_\lambda^n h(\lambda)
=
\sum_{m=0}^n \binom{n}{m} 
(-1)^{n-m} h((m+1)\lambda),
\quad
n\in \N.
\end{align*} 
Let $\eta_0$ and $\s$ be 
the positive number and the mapping
appearing in Proposition \ref{prop:S},
respectively.
For $\phi \in \mh$ and $\lambda>0$ 
with $\lambda \phi \in B_{\eta_0}\mh$,
we set 
\begin{align}\label{def:K}
K^{\phi,\lambda}
&=
\bra{
\s(\lambda\phi) - \lambda\phi,
\Miijj{0}{-\omega^{-2}}{1}{0}\phi
}_{\mh}.
\end{align} 
We remark that 
if $n\ge 2$, 
then 
\begin{align*}
\varDelta_\lambda^n
K^{\phi,\lambda}
&=
\varDelta_\lambda^n
\bra{
\s(\lambda\phi),
\Miijj{0}{-\omega^{-2}}{1}{0}\phi
}_{\mh}.
\end{align*} 
For $\phi\in \mh$ and $y_3,y_4 \in \R$,
we set functions 
$\wW^\phi_3[y_3]=y_3 w_\phi^3$ 
and
$\wW^\phi_4[y_3,y_4]=y_4 w_\phi^4$. 
We also define inductively 
for $\phi\in \mh$, 
$N=5,6,\cdots$,
and $y_3,y_4,\cdots, y_N \in \R$,
\begin{align}\label{def:W}
W^\phi_N[y_3,\cdots,y_{N-2}]
&=
\sum_{k=1}^{k_N}
\frac{N!}{k!}
\sum_{
 \substack{ 
 N_0+N_1+\cdots+N_k=N
 \\ 
 N_0\ge (3-k)_+
 \\
 N_1,\cdots,N_k\ge 3
 }
}
\frac{y_{N_0+k} w_\phi^{N_0} }{N_0!}
\prod_{\ell=1}^k
\frac{
\Gs
\widetilde{W}^\phi_{N_\ell}[y_3,\cdots,y_{N_\ell}]
}{ N_\ell ! },
\nonumber
\\
\wW^\phi_N[y_3,\cdots,y_N]
&=
y_N w_\phi^N
+
W^\phi_N[y_3,\cdots,y_{N-2}],
\end{align} 
where 
$k_N=\lfloor (N-3)/ 2 \rfloor \wedge \lfloor N/3 \rfloor$, 
which is the largest natural number $k$ 
such that  
$(3-k)_+ + 3k \le N$, 
and 
$\lfloor\,\cdot \,\rfloor$ is the floor function.
Remark that the above 
$W^\phi_N[y_3,\cdots,y_{N-2}]$
is actually defined 
only by $\phi,N$ and $y_3,\cdots,y_{N-2}$. 
Indeed, 
$N_\ell$ and $N_0+k$ in (\ref{def:W}) 
are not larger than $N-2$ since 
\begin{align*}
N-N_\ell
&=
N_0+N_1+\cdots+N_k-N_\ell
\ge 
(3-k)_+ + 3(k-1) \ge 2,
\\
N-(N_0+k)
&=
N_1+\cdots+N_k -k
\ge 
3k-k \ge 2.
\end{align*}

\begin{rem}\label{rem:def W}
\begin{enumerate}[(1)]
  \item 
We see from the proof 
of Proposition \ref{prop:N} below that 
all of functions 
$\widetilde{W}^\phi_N[y_3,\cdots,y_N]$ 
and 
$W^\phi_N[y_3,\cdots,y_{N-2}]$  
belong to $C^0 L^2 \cap \aY$.
  \item 
For example, 
we have
\begin{align*}
W^\phi_5[y_3]
=
10y_3^2 w_\phi^2 \Gs w_\phi^3,
\quad
W^\phi_6[y_3,y_4]
=
y_3y_4
\braa{
20 w_\phi^3 \Gs w_\phi^3
+
15 w_\phi^2 \Gs w_\phi^4
}
\end{align*}
and 
\begin{align*}
W^\phi_7[y_3,y_4,y_5]
&=
y_3^3 \braa{
210 w_\phi^2 \Gs \braa{ w_\phi^2 \Gs w_\phi^3}
+
70 w_\phi        \braa{ \Gs w_\phi^3}^2 
}
+
35 y_4^2 
w_\phi^3 \Gs w_\phi^4
\\
&\quad +
y_3y_5 \braa{
35 w_\phi^4 \Gs w_\phi^3 
+
21 w_\phi^2 \Gs w_\phi^5
}.
\end{align*} 
\end{enumerate}
\end{rem}

Recall that $K^{\phi,\lambda}$ is 
the given value defined by (\ref{def:K}).  
We are ready to state our main results.
\begin{theorem}\label{thm:main}
Put $N\in \N$ with $N\ge 3$.
Choose $\phi \in \mh$ so that
\begin{align}\label{nonzero}
\int_{\R^{1+2}}
w_\phi^{N+1}
d(t,x)
\neq 0.   
\end{align} 
\begin{enumerate}[(1)]
  \item 
Assume $N=3,4$.
Then 
there exists some positive number 
$\Lambda_N$ such that  
for any 
$\lambda \in (0, \Lambda_N)$,
\begin{align*}
 \abs{
\n^{(N)}(0)
-
\frac{
\displaystyle{
\lambda^{-N} \varDelta_\lambda^N
K^{\phi,\lambda}
}
}{
\displaystyle{
\int_{\R^{1+2}}
w_\phi^{N+1}
d(t,x)
}
}
 }
\lesssim_{\phi}
\lambda.
\end{align*}
In particular, 
the unknown value of $\n^{(N)}(0)$ 
is uniquely determined by  
\begin{align*}
\n^{(N)}(0)
=
\frac{
\displaystyle{
\lim_{\lambda\to +0}
\lambda^{-N} \varDelta_\lambda^N
K^{\phi,\lambda}
}
}{
\displaystyle{
\int_{\R^{1+2}}
w_\phi^{N+1}
d(t,x)
}
}.
\end{align*} 
  \item 
Assume $N\ge 5$ and that 
we have already known the values of 
$\n^{(3)}(0),\cdots,\n^{(N-2)}(0)$. 
Then 
there exists some positive number 
$\Lambda_N$ such that  
for any 
$\lambda \in (0, \Lambda_N)$,
\begin{align*}
&
 \abs{
\n^{(N)}(0)
-
\frac{
\displaystyle{
\lambda^{-N} \varDelta_\lambda^N
K^{\phi,\lambda}
-
\int_{\R^{1+2}}
W^\phi_N
\brac{
\n^{(3)}(0),\cdots,\n^{(N-2)}(0)
}
w_\phi
d(t,x)
}
}{
\displaystyle{
\int_{\R^{1+2}}
w_\phi^{N+1}
d(t,x)
}
}
 }
\\
& 
\lesssim_{N,\phi}
\lambda.
\end{align*} 
In particular, 
the unknown value of $\n^{(N)}(0)$ 
is uniquely determined by  
\begin{align*}
&\n^{(N)}(0)
=
\frac{
\displaystyle{
\lim_{\lambda\to +0}
\lambda^{-N} \varDelta_\lambda^N
K^{\phi,\lambda}
-
\int_{\R^{1+2}}
W^\phi_N
\brac{
\n^{(3)}(0),\cdots,\n^{(N-2)}(0)
}
w_\phi
d(t,x)
}
}{
\displaystyle{
\int_{\R^{1+2}}
w_\phi^{N+1}
d(t,x)
}
}.
\end{align*}
\end{enumerate} 
\end{theorem}

\begin{rem}\label{rem:1:thm:main:1}
For any $N\in \N$ with $N\ge 3$,
we can actually choose a function $\phi\in \mh$ so that 
$
\int_{\R^{1+2}}
w_\phi^{N+1}
d(t,x)
\neq 0.
$
For more details,
see Appendix B below.
\end{rem}

Even if the values of  
$\n^{(3)}(0),\cdots,\n^{(N-1)}(0)$ 
($N\ge 4$)
are unknown,
we can directly identify $\n^{(N)}(0)$ 
by the following theorem:
\begin{theorem}\label{thm:main:2}
For $\lambda\in (0,1]$,  
we set  
$I^\lambda_n \in \R$ 
$(n=3,4,\cdots)$
inductively as follows:
\begin{itemize}
  \item 
For $n=3,4$,
we choose $\phi[n] \in B_{\eta_0}\mh$ so that 
\begin{align}\label{nonzero}
\int_{\R^{1+2}}
w_{\phi[n]}^{n+1}
d(t,x)
\neq 0
\end{align} 
and put 
\begin{align*}
I^{\lambda}_n
&=
\frac{
\displaystyle{
\lambda^{-n} \varDelta_\lambda^n
K^{\phi[n],\lambda}
}
}{
\displaystyle{
\int_{\R^{1+2}}
w_{\phi[n]}^{n+1}
d(t,x)
}
}.
\end{align*} 
  \item 
For $n=5,6,\cdots$,
we choose $\phi_n \in B_{\eta_0}\mh$ satisfying  
(\ref{nonzero}), 
and put 
\begin{align*}
I^{\lambda}_n
=
\frac{
\displaystyle{
\lambda^{-n} \varDelta_\lambda^n
K^{\phi[n],\lambda}
-
\int_{\R^{1+2}}
W^{\phi[n]}_n
\brac{
I^{\lambda}_3,\cdots,I^{\lambda}_{n-2}
}
w_{\phi[n]}
d(t,x)
}
}{
\displaystyle{
\int_{\R^{1+2}}
w_{\phi[n]}^{n+1}
d(t,x)
}
}
.
\end{align*} 
\end{itemize}
Then for any $N\in \N$ with $N\ge 3$, 
there exists some positive number 
$\Lambda_N^\prime$ such that  
\begin{align}\label{est:thm:main:2}
 \abs{
\n^{(n)}(0)
-
I^\lambda_n
 }
\lesssim_{N,\phi[3],\cdots, \phi[N] }
\lambda,
\quad
\lambda \in (0, \Lambda_N^\prime)
\quad
(3 \le n \le N).
\end{align} 
In particular, 
the unknown value of $\n^{(N)}(0)$ 
is uniquely determined by  
\begin{align*}
\n^{(N)}(0)
=
\lim_{\lambda\to +0}
I^{\lambda}_N.
\end{align*} 
\end{theorem}

\begin{rem}\label{rem:1:thm:main:2}
For example, 
by the definition of $I^\lambda_N$ 
we obtain 
\begin{align*}
I^\lambda_5
&=
\frac{1}{
\displaystyle{
\int_{\R^{1+2}}
w_{\phi[5]}^6
d(t,x)
}}
\\
&\quad \times
  \brab{
\lambda^{-5} \varDelta_\lambda^5 K^{\phi[5],\lambda}
-
\braa{
\lambda^{-3} \varDelta_\lambda^3 K^{\phi[3],\lambda}
}^2
\frac{
\displaystyle{
10\int_{\R^{1+2}}
w_{\phi[5]}^3 \Gs w_{\phi[5]}^3
d(t,x)
}
}{
\displaystyle{
\braa{
\int_{\R^{1+2}}
w_{\phi[3]}^4
d(t,x)
}^2
}
}
  }.
\end{align*} 
\end{rem}

\begin{rem}\label{rem:2:thm:main:2}
Recall functions $\n_P$ and $\n_{c_1,c_2}$ 
mentioned at the beginning of this section.
\begin{enumerate}[(1)]
  \item 
In the case $\n=\n_P$, 
we can uniquely reconstruct $\n$ 
since 
\begin{align*}
\n_P^{(N)}(0)
=\frac{N!}{(N-3)!} P^{(N-3)}(0)
\quad (N\ge 3).
\end{align*}
  \item 
Since 
$\n_{c_1,c_2}^{(3)}(0)=6c_1c_2$
and 
$\n_{c_1,c_2}^{(5)}(0)=60c_1c_2^2$,
if $\n=\n_{c_1,c_2}$
then $c_1,c_2$ are uniquely determined. 
\end{enumerate}
\end{rem}

\subsection{G\^{a}teaux differentials of $\s$.}
Let us go back to the direct scattering problem.
The functions 
$\wW_N^\phi[y_3,\cdots,y_N]$ 
mentioned above can be used to 
express higher order G\^{a}teaux differentials 
of $\s$ at 0.
In this paper, 
we define 
higher order G\^{a}teaux differentiability 
as follows (see, e.g., 
\cite{Gateaux1919,HillePhillips1974}): 
For Banach spaces 
$\mathcal{X},\mathcal{Y}$,  
an open set $U \subset \mathcal{X}$, 
a mapping $F:U\to \mathcal{Y}$ 
and $N\in \N$, 
we say that 
$F$ is $N$-th order G\^{a}teaux differentiable  
as a mapping from $U$ to $\mathcal{Y}$ 
if for any 
$\varphi_0 \in U$
and 
$\varphi_1 \in \mathcal{X}$, 
the vector 
\begin{align*}
d^N F(\varphi_0;\varphi_1)
:=
\left.
\frac{d^N}{d \lambda^N}
F(\varphi_0+\lambda \varphi_1)
\right|_{\lambda=0}
\end{align*} 
exists in the $\mathcal{Y}$-sense.
We call the above $d^N F(\varphi;\phi)$ 
the $N$-th order G\^{a}teaux differential 
of $F$ at $\varphi_0$ in the direction $\varphi_1$.
We now introduce 
an expression of 
higher order G\^{a}teaux differentials 
of $\s$ at 0.

\begin{theorem}\label{thm:main:3}
Fix $N \in \N$.
Then for some open set 
$U_N \subset \mh$ containing 0, 
the scattering operator 
$\s$ is $N$-th order G\^{a}teaux differentiable 
as a mapping from $U_N$ to $\mh$.
Furthermore, 
it follows that for any $\phi\in \mh$,
\begin{align*}
d^N \s(0;\phi)
=
\left\{
  \begin{array}{cl}
\phi&(N=1),\\
0   &(N=2),\\
\displaystyle{
\int_{\R}
\binom{
-\oi\sin(t\omega) 
}{
 \cos(t\omega)   
}
\wW_N^\phi[\n^{(3)}(0),\cdots,\n^{(N)}(0)]
dt
}
&
     (N\ge 3).\\
  \end{array}
\right.
\end{align*} 
\end{theorem}

\begin{rem}\label{rem:1:thm:main:3}
For example,
we set $a\in \R$ and $\n(y)=ay^3/6$.
Then we have 
\begin{align}\label{id:rem:1:thm:main:3}
d^N \s(0;\phi)
=
\left\{
  \begin{array}{cl}
\phi
&\text{if $N=1$,}\\
0
&\text{if $N$ is even,}\\
\displaystyle{
a^{(N-1)/2}
\int_{\R}
\binom{
-\oi\sin(t\omega) 
}{
 \cos(t\omega)   
}
\mathcal{W}_N^\phi
dt
}
&\text{if $N$ is odd and $N\ge 3$.}
  \end{array}
\right.
\end{align} 
Here, 
we have defined inductively as follows:
\begin{itemize}
  \item 
We put $\mathcal{W}_3^\phi=w_\phi^3$.
  \item 
For $n\in \N$, 
we put 
\begin{align*}
\mathcal{W}_{2n+3}^\phi
=
\sum_{k=1}^{n \wedge 3}
\frac{(2n+3)!}{k!(3-k)!}
\sum_{
 n_1+\cdots+n_k=n-k
}
w_\phi^{3-k}
\prod_{\ell=1}^k
\frac{
\Gs
\mathcal{W}_{2n_\ell+3}^\phi
}{ (2n_\ell +3) ! }.
\end{align*} 
\end{itemize} 
We also list some of 
$\mathcal{W}_{2n+1}^\phi$:
\begin{align*}
\mathcal{W}_5^\phi
&=
10 w_\phi^2 \Gs w_\phi^3,
\\
\mathcal{W}_7^\phi
&=
70 w_\phi (\Gs w_\phi^3)^2 
+
210 w_\phi^2 \Gs 
\left(w_\phi^2 \Gs w_\phi^3\right),
\\
\mathcal{W}_9^\phi
&=
280 (\Gs w_\phi^3)^3
+
5040 w_\phi 
 \left( \Gs w_\phi^3 \right)
 \Gs \left( w_\phi^2 \Gs w_\phi^3 \right) 
\\
&\quad +
2520 w_\phi^2 \Gs 
\left(
 w_\phi (\Gs w_\phi^3)^2
\right)
+
7560 w_\phi^2 \Gs 
\left(
 w_\phi^2 
 \Gs\left(w_\phi^2 \Gs w_\phi^3 \right)
\right)
.
\end{align*}
We show (\ref{id:rem:1:thm:main:3}) 
in Appendix A below.
\end{rem}

We finally mention the contents of the rest of this paper.
In Section 2, 
we note
some inequalities associated with 
solutions to linear Klein-Gordon equations 
and Besov spaces.
In Sections 3, 
we give estimates 
for higher-order expansions 
of the scattering operator.
In Section 4, 
we establish Theorems 
\ref{thm:main},
\ref{thm:main:2} and \ref{thm:main:3}.
Identity (\ref{id:rem:1:thm:main:3}) is proved 
in Appendix A. 
The existence of the function $\phi$ mentioned in Remark \ref{rem:1:thm:main:1} 
is shown in Appendix B.

\section{Preliminaries}
We introduce some properties used in Section 3.
For this purpose,
we list some notations.
Put 
$Y=L^4 B^{1/2}_4$,
$\wZ=\li \cap Y$ 
and 
\begin{align*}
\nr{v}_{\wZ}=\nr{v}_{\li} \vee \nr{v}_{Y},
\quad
v\in \wZ.
\end{align*} 
For 
an interval $I$ of $\R$, 
$N\in \N$, 
and a mapping 
$V:I\ni \lambda \mapsto V(\lambda) \in \aY$ 
in $C^N(I;\aY)$,
by $\Dl^N V$ 
we denote the $N$-th strong derivative 
in $\aY$.
\\

Henceforth, 
we often use
the following type of Taylor's theorem:

\begin{prop}\label{prop:Taylor}
Put $n\in \Z$.
Let $I$ be an open interval of $\R$.
Assume $h\in C^{n+1}(I;\R)$.
Then 
it follows that 
for any $y_0,y_1 \in I$,
\begin{align*}
h(y_0+y_1)
=
\sum_{k=0}^n \frac{1}{k!}h^{(k)}(y_0) y_1^k
+
\int_0^1
\frac{(1-\theta)^n}{n!} 
h^{(n+1)}(y_0 + \theta y_1) 
y_1^{n+1}
d\theta.
\end{align*} 
\end{prop}

Recall that $\mh$ and $Z$ are function spaces mentioned above.
By Strichartz type estimates 
(see, e.g., 
\cite{NakamuraOzawa2001,Wang1999}), 
we obtain the following proposition:
\begin{prop}\label{prop:St}
The mapping 
$\phi \mapsto  w_\phi$
(resp. $G    \mapsto  \Gs G$) 
is a bounded linear operator 
from $\mh$ to $Z$ 
(resp. $\aY$ to $Z$). 
\end{prop}

We next give two propositions 
associated with the function $\n$.

\begin{prop}\label{prop:N} 
Some $(\varepsilon_k)_{k=0}^\infty \subset (0,1]$ 
and $\varepsilon>0$  
satisfy the following properties:
\begin{enumerate}[(1)]
  \item
If 
$N\in \Z$, 
$v \in \wZ \cap B_{\varepsilon_N}\li$
and $\phi\in \mh$,
then 
\begin{align*}
\nr{\n^{(N)}(v) w_\phi^N}_{\aY} 
\lesssim_N 
\braa{
\nr{v}_{\wZ}^{(3-N)_+}
+
\nr{v}_{\wZ}^{1+(2-N)_+}
}
\nr{\phi}^N.
\end{align*} 
  \item 
If $k\in \N$, 
$N\in \Z$, 
$v \in \wZ \cap B_{\varepsilon_{k+N}}\li$ 
and 
$G_1,\cdots,G_k \in \aY$, 
then 
\begin{align*}
&
\nr{
\n^{(k+N)}(v) w_\phi^N 
\prod_{\ell=1}^k \Gamma G_\ell
}_{\aY} 
\\
&\lesssim_{k,N}
\braa{
\nr{v}_{\wZ}^{(3-k-N)_+}
+
\nr{v}_{\wZ}^{1+(2-k-N)_+}
}
\nr{\phi}^N
\prod_{\ell=1}^k
\nr{G_\ell}_{\aY}.
\end{align*} 
  \item 
If 
$v \in C^0 H^1 \cap B_{\varepsilon}\li$,
then 
$\n(v) \in C^0 L^2$.
\end{enumerate}
\end{prop}

In order to show the above proposition,
we list some notation and lemma.
For $z\in \R^2$ and for a function $f:\R^2 \to \R$, 
we set 
$(T_z f)(x)=f(x+z)$ ($x\in \R^2$)
and 
$\vD f=T_z f-f$.
For $\theta\in [0,1]$ and $a,b\in \R$, 
we put 
$[a,b]_\theta=\theta a +(1-\theta)b$.
The following lemma is easily proved 
by standard properties of Besov spaces
(see, e.g., \cite{OgawaOzawa1991}):

\begin{lem}\label{lem:prop:N}
There exists a constant $C_E \ge 1$ 
satisfying the following properties:
\begin{enumerate}[(1)]
  \item 
$\nr{f}_r \le C_E \sqrt{r} \nr{f}_{B_4^{1/2}}$
$(r\ge 4, \ f\in B_4^{1/2})$.
  \item 
$\nr{f}_{B_4^{1/2}} \le C_E \nr{f}_{H^1}$
$(f\in H^1)$.
  \item 
It follows that 
for any $r\in [1,\infty]$ and $f\in B_r^{1/2}$,
\begin{align*}
\nr{f}_{B_r^{1/2}}
\le 
C_E\nr{f}_r
+
C_E \sqrt{
 \int_0^\infty
 \braa{
 y^{-1/2} \sup_{|z|\le y} \nr{\vD f}_r
 }^2
 \frac{dy}{y}
}
\le 
C_E^2 \nr{f}_{B_r^{1/2}}.
\end{align*} 
\end{enumerate}
\end{lem}

\begin{proof}[Proof of Proposition \ref{prop:N}]
Since (1) can be shown more easily than (2), 
we prove only (2) and (3).
We see from the hypothesis of $\n$ that 
some $(c_k)_{k=0}^\infty \subset (0,\infty)$ satisfies 
\begin{align*}
\abs{\n^{(k)}(y)}
\le
c_k |y|^{(3-k)_+} e^{c_k y^2}
=
\sum_{m=0}^\infty
\frac{c_k^{m+1}}{m!} |y|^{2m+(3-k)_+},
\quad
k\in \Z,\ y\in \R.
\end{align*} 
By Proposition \ref{prop:Taylor},
we obtain for any $v\in \aY$,
\begin{align}\label{ineq:1:prf:prop:N}
\abs{
\vD \n^{(k)}(v)
}
&\le 
\int_0^1
\abs{
\n^{(k+1)}([T_z v,v]_\theta)
\vD v
}
d\theta
\nonumber
\\
&=
\int_0^1
\sum_{m=0}^\infty
\frac{c_{k+1}^{m+1}}{m!} 
\abs{ [T_z v,v]_\theta }^{2m+(2-k)_+}
\abs{ \vD v}
d\theta,
\quad
k\in \Z,\ z\in \R^2.
\end{align}

We now prove (2).
Fix $k\in \N$.
Put 
$\varepsilon_k>0$, 
$v \in \wZ \cap B_{\varepsilon_k}\li$
and 
$v_1,\cdots,v_k \in \wZ$.
We see from Proposition \ref{prop:St} 
that it is sufficient to establish 
\begin{align}\label{ineq:2:prf:prop:N}
\nr{
\n^{(k)}(v) 
\prod_{\ell=1}^k v_\ell
}_{\aY} 
\lesssim_k 
\braa{
\nr{v}_{\wZ}^{(3-k)_+}
+
\nr{v}_{\wZ}^{1+(2-k)_+}
}
\prod_{\ell=1}^k
\nr{v_\ell}_{\wZ}.
\end{align} 
Put $z\in \R^2$.
Then we have
\begin{align}\label{id:1:prf:prop:N}
\vD \braa{ 
\n^{(k)}(v) 
\prod_{\ell=1}^k v_\ell
}
=
\braa{
\vD \n^{(k)}(v) 
}
\prod_{\ell=1}^k v_\ell
+
\braa{
T_z \n^{(k)}(v)
}
\vD \prod_{\ell=1}^k v_\ell
=:
I_1+I_2.
\end{align} 
For $m\in \Z$, 
set $r(m,k)=4m+2k+2(2-k)_+$.
It follows from 
(\ref{ineq:1:prf:prop:N}), 
the H\"older inequality 
and Lemma \ref{lem:prop:N}
that 
\begin{align*}
&
\nr{ I_1 }_{4/3}
\\
&\le 
\sum_{m=0}^\infty
\frac{c_{k+1}^{m+1}}{m!} 
\nr{ v }^{2m+(2-k)_+}_{r(m,k)}
\nr{ \vD v}_4
\prod_{\ell=1}^k 
\nr{v_\ell}_{r(m,k)}
\\
&\le
\sum_{m=0}^\infty
\frac{c_{k+1}^{m+1}}{m!} 
\braa{C_E \sqrt{r(m,k)} }^{2m+k+(2-k)_+}
\nr{ v }^{2m+(2-k)_+}_{B_4^{1/2}}
\nr{ \vD v}_4
\prod_{\ell=1}^k 
\nr{v_\ell}_{B_4^{1/2}}
\\
&\le
C_E^{k-2+(2-k)_+}
\nr{ \vD v}_4
\nr{ v }^{(2-k)_+}_{B_4^{1/2}}
\braa{
\prod_{\ell=1}^{2-(2-k)_+}
\nr{v_\ell}_{B_4^{1/2}}
}
\braa{
\prod_{\ell=3-(2-k)_+}^k 
\nr{v_\ell}_{H^1}
}
\\
&\quad\times
\sum_{m=0}^\infty
\frac{c_{k+1}^{m+1}}{m!} 
C_E^{4m+k+(2-k)_+}
r(m,k)^{r(m,k)/4}
\nr{ v }^{2m}_{H^1}.
\end{align*} 
Since 
\begin{align*}
r(m,k)^{r(m,k)/4} 
\le
(4m+4k)^{m+k}
\lesssim_k
(4e)^m (m+k)!
\lesssim_k
(8e)^m m!,
\end{align*} 
we obtain
\begin{align*}
\nr{ I_1 }_{4/3}
&\lesssim_k
\nr{ \vD v}_4
\nr{ v }^{(2-k)_+}_{B_4^{1/2}}
\braa{
\prod_{\ell=1}^{2-(2-k)_+}
\nr{v_\ell}_{B_4^{1/2}}
}
\braa{
\prod_{\ell=3-(2-k)_+}^k 
\nr{v_\ell}_{H^1}
}
\\
&\quad \times
\sum_{m=0}^\infty
\braa{
8c_{k+1}C_E^4 e
}^m 
\nr{ v }^{2m}_{H^1}.
\end{align*} 
Similarly,
we have 
\begin{align*}
\nr{
\n^{(k)}(v) 
\prod_{\ell=1}^k v_\ell
}_{4/3}
\lesssim_k
\nr{ v_1 }_4
\nr{ v }^{(3-k)_+}_{B_4^{1/2}}
\braa{
\prod_{\ell=2}^{k}
\nr{v_\ell}_{\mathcal{X}_{k,1,\ell}}
}
\sum_{m=0}^\infty
\braa{
8c_k C_E^4 e
}^m 
\nr{ v }^{2m}_{H^1}
\end{align*} 
and 
\begin{align*}
\nr{I_2}_{4/3}
\lesssim_k
\sum_{j=1}^k
\nr{ \vD v_j }_4
\nr{ v }^{(3-k)_+}_{B_4^{1/2}}
\braa{
\prod_{
\substack{
1\le \ell\le k
\\
\ell \neq j
}
}
\nr{v_\ell}_{\mathcal{X}_{k,j,\ell}}
}
\sum_{m=0}^\infty
\braa{
8c_k C_E^4 e
}^m 
\nr{ v }^{2m}_{H^1}.
\end{align*} 
Here, 
for $k,j,\ell \in \N$ 
we have defined 
\begin{align*}
\mathcal{X}_{k,j,\ell}
=
\left\{
  \begin{array}{cl}
B_4^{1/2}&(k=2, \, j=1,2, \, \ell\neq j),    \\
B_4^{1/2}&(k\ge 3, \, j=1, \, \ell=2,3),\\
B_4^{1/2}&(k\ge 3, \, j=2, \, \ell=1,3),\\
B_4^{1/2}&(k\ge 3, \, 3\le j \le k, \, \ell=1,2),\\
H^1      &(otherwise).\\
  \end{array}
\right.
\end{align*} 
We see from 
(\ref{id:1:prf:prop:N}) and Lemma \ref{lem:prop:N} that 
(\ref{ineq:2:prf:prop:N}) holds true 
if $0<\varepsilon_k\ll 1$.

We next show (3). 
Fix $t,t_0\in I$.
Then by 
Proposition \ref{prop:Taylor},
we obtain 
\begin{align*}
\abs{
\n(v(t))-\n(v(t_0))
}
&\le 
\int_0^1 
\abs{
\n^{(1)}([v(t),v(t_0)]_\theta) 
}
\abs{
v(t)-v(t_0)
}
d\theta
\\
&\le
\int_0^1 
\sum_{m=0}^\infty
\frac{c_1^{m+1}}{m!}
\abs{
[v(t),v(t_0)]_\theta
}^{2m+2}
\abs{
v(t)-v(t_0)
}
d\theta.
\end{align*} 
As in the proof of (2),
we have 
\begin{align*}
\nr{
\n(v(t))-\n(v(t_0))
}_2
\le
\nr{
v(t)-v(t_0)
}_{H^1}
\sum_{m=0}^\infty
\frac{c_1^{m+1}C_E^{4m+6}}{m!}
(4m+6)^{m+2}
\nr{
v}_{\li}^{2m+2}.
\end{align*} 
Therefore,
if $\nr{v}_{\li}$ is sufficiently small, 
then $\n(v)\in C^0 L^2$.
\end{proof}

Recall that $D_\lambda$ is 
the strong derivative in $\lambda$ mentioned above.

\begin{prop}\label{prop:DN} 
Put $\phi\in \mh$ and $N\in \Z$.
Let $\varepsilon_k$ $(k\in \Z)$
be the number appearing in Proposition \ref{prop:N}.
\begin{enumerate}[(1)]
  \item 
If $\delta>0$ and  
$\sqrt{2}\delta \nr{\phi} \le \varepsilon_{N+2}$,
then 
$
\n^{(N)}(\lambda w_\phi)w_\phi^N 
\in C^1((-\delta,\delta);\aY)
$
and 
\begin{align*}
\Dl 
\braa{
\n^{(N)}(\lambda w_\phi)w_\phi^N
}=
\n^{(N+1)}(\lambda w_\phi)w_\phi^{N+1}
\ 
\text{in $\aY$ $(\lambda\in (-\delta,\delta))$.}
\end{align*}
  \item 
Put $k\in \N$. 
Assume that 
$
G_1(\lambda),\cdots,G_k(\lambda)\in 
C^1((-\delta^\prime,\delta^\prime);\aY)
$
for some $\delta^\prime>0$.
\end{enumerate}
If 
$0<\delta \le \delta^\prime$
and 
$
\sqrt{2}\delta \nr{\phi} \le 
\varepsilon_{k+N} 
\wedge \varepsilon_{k+N+1}
\wedge \varepsilon_{k+N+2}
$,
then
\begin{align*}
\n^{(k+N)}(\lambda w_\phi)w_\phi^N 
\prod_{\ell=1}^k \Gs G_\ell(\lambda)
\in C^1((-\delta,\delta);\aY)
\end{align*} 
and 
\begin{align}
&
\Dl 
\braa{
\n^{(k+N)}(\lambda w_\phi)w_\phi^N 
\prod_{\ell=1}^k \Gs G_\ell(\lambda)
}
=
\n^{(k+N+1)}(\lambda w_\phi)w_\phi^{N+1}
\prod_{\ell=1}^k \Gs G_\ell(\lambda)
\nonumber
\\
&\quad +
\n^{(k+N)}(\lambda w_\phi)w_\phi^N 
\sum_{\ell=1}^k \Gs \Dl G_\ell(\lambda)
\prod_{\substack{1\le m \le k \\ m\neq \ell}} 
\Gs G_m(\lambda)
\ 
\text{in $\aY$ $(\lambda\in (-\delta,\delta))$.}
\label{id:prop:DN}
\end{align} 
\end{prop}

\begin{proof}[Outline of proof]
We show only (\ref{id:prop:DN}) with $k=1$ 
since other cases are proved similarly.
Assume 
$0<\delta\le \delta^\prime$ and 
$
\sqrt{2} \delta \nr{\phi} \le 
\varepsilon_{N+1} 
\wedge \varepsilon_{N+2}
\wedge \varepsilon_{N+3}
$.
By direct calculation, 
we have 
for any $\lambda,\lambda_0 \in (-\delta,\delta)$
with $\lambda\neq \lambda_0$,
\begin{align*}
&
\frac{
\n^{(N+1)}(\lambda w_\phi)w_\phi^N 
\Gs G_1(\lambda)
-
\n^{(N+1)}(\lambda_0 w_\phi)w_\phi^N 
\Gs G_1(\lambda_0)
}{\lambda-\lambda_0}
\\
&\quad -
\n^{(N+2)}(\lambda_0 w_\phi)w_\phi^{N+1}
\Gs G_1(\lambda_0)
-
\n^{(N+1)}(\lambda_0 w_\phi)w_\phi^N 
\Gs \Dl G_1(\lambda_0)
\\
&=
\braa{
\frac{
\n^{(N+1)}(\lambda w_\phi)w_\phi^N 
-
\n^{(N+1)}(\lambda_0 w_\phi)w_\phi^N 
}{\lambda-\lambda_0}
-
\n^{(N+2)}(\lambda_0 w_\phi)w_\phi^{N+1}
}
\Gs G_1(\lambda_0)
\\
&\quad +
\n^{(N+1)}(\lambda_0 w_\phi)w_\phi^N 
\Gs
\braa{
\frac{
G_1(\lambda)-G_1(\lambda_0)
}{
\lambda-\lambda_0
}
-\Dl G_1(\lambda_0)
}
\\
& \quad +
\braa{
\n^{(N+1)}(\lambda w_\phi)w_\phi^N 
-
\n^{(N+1)}(\lambda_0 w_\phi)w_\phi^N 
}
\Gs
\frac{
G_1(\lambda)-G_1(\lambda_0)
}{
\lambda-\lambda_0
}
\\
&=F_1+F_2+F_3.
\end{align*} 
We see from Proposition \ref{prop:Taylor} 
that 
\begin{align*}
\nr{ F_1 }_{\aY}
\le 
|\lambda-\lambda_0|
\int_0^1 
\nr{
\n^{(N+3)}\braa{
[
\lambda,\lambda_0
]_\theta w_\phi
}
w_\phi^{N+2} 
\Gs G_1(\lambda_0)
}_{\aY}
d\theta.
\end{align*} 
Since 
\begin{align*}
\nr{
[
\lambda,\lambda_0
]_\theta w_\phi
}_{\li}
\le 
\delta \nr{ w_\phi }_{\li}
\le
\sqrt{2} \delta \nr{\phi}
\le
\varepsilon_{N+3},
\end{align*} 
it follows from Proposition 
\ref{prop:N} that 
\begin{align*}
\nr{ F_1 }_{\aY}
\lesssim_N
|\lambda-\lambda_0|
\braa{
1+\nr{\phi}
}
\nr{\phi}^{N+2} \nr{ G_1(\lambda_0) }_{\aY}.
\end{align*} 
Similarly, 
we obtain 
\begin{align*}
\nr{ F_2 }_{\aY}
\lesssim_N
\braa{
\nr{\phi}^{(2-N)_+}+\nr{\phi}^{1+(1-N)_+}
}
\nr{\phi}^N 
\nr{ 
\frac{
G_1(\lambda)-G_1(\lambda_0)
}{
\lambda-\lambda_0
}
-\Dl G_1(\lambda_0)
}_{\aY}.
\end{align*} 
Hence we have 
$
\lim_{\lambda\to\lambda_0}F_1
=\lim_{\lambda\to\lambda_0}F_2
=0
$
in $\aY$. 
We can show 
$\lim_{\lambda\to\lambda_0}F_3=0$ in $\aY$ 
more easily,
which completes the proof.
\end{proof}

\section{Expansion of ${\mbox{\boldmath $S$}}$}
Our goal in the present section is 
to show a higher-order expansion of 
the scattering operator $\s$, 
which is essential to establish our main results.
We see from Proposition \ref{prop:S} that 
for any $\phi\in \mh$ and $\lambda \in \R$ 
with 
$\lambda \phi \in B_{\eta_0}\mh$, 
there exists a unique time-global solution 
$u_{\lambda \phi} \in Z$ 
to the integral equation of the form 
\begin{align*}
u_{\lambda \phi} 
= 
w_{\lambda \phi} +\Gs \n(u_{\lambda \phi}). 
\end{align*}
Furthermore, 
we have $\n(u_{\lambda \phi}) \in \aY$,
\begin{align}\label{est:u}
\nr{\n(u_{\lambda \phi})}_{\aY} 
\lesssim \nr{\lambda \phi}^3
\end{align} 
and 
\begin{align}\label{id:S}
\s(\lambda\phi)-\lambda\phi
=
\int_{\R}
\binom{
-\oi\sin(t\omega) 
}{
 \cos(t\omega)   
}
\n(u_{\lambda \phi}(t))
dt.
\end{align} 
Then we obtain  
\begin{align}\label{id:K}
K^{\phi,\lambda}
=
\int_{\R} 
\bra{\n(u_{\lambda\phi}(t)),w_\phi(t)}_{L^2}dt, 
\quad 
\phi\in \mh, \
0<\lambda \ll 1 
\end{align} 
by using 
Propositions \ref{prop:St} and \ref{prop:N}
and 
the following proposition:

\begin{prop}\label{prop:K}
Set $\psi\in \mh$ and $G\in \aY \cap C^0 L^2$.
Then it follows that 
\begin{align*}
\bra{
\int_{\R}
\binom{
-\oi\sin(t\omega) 
}{
 \cos(t\omega)   
}
G(t)
dt,
\Miijj{0}{-\omega^{-2}}{1}{0}\psi
}_{\mh}
=
\int_{\R} 
\bra{G(t),w_\psi(t)}_{L^2}dt.
\end{align*} 
\end{prop}

\begin{proof}
Fix $G\in \aY \cap C^0L^2$.
Then time decay estimates 
and 
Strichartz type estimates 
imply 
\begin{align}\label{L1:prf:prop:K}
\bra{x}^{-1}\omega^{-1/2}
\binom{
-\sin(t\omega)  
}{
 \cos(t\omega)  
}
G(t)
\in L^1(\R_t; L^2 \oplus L^2)
\end{align} 
and 
\begin{align*}
\int_{\R}
\binom{
-\oi\sin(t\omega) 
}{
 \cos(t\omega)   
}
G(t)
dt
\in \mh,
\end{align*}
respectively. 
For the proof of (\ref{L1:prf:prop:K}), 
see Proposition 3.3 in \cite{Sasaki2024}.

We have only to prove the case 
$\psi\in \Sc \oplus \Sc$ 
since 
$\Sc\oplus\Sc$ is dense in $\mh$ 
and we have the estimate 
\begin{align*}
\nr{
G w_{\psi_1}
-
G w_{\psi_2} 
}_{L^1L^1}
&\le 
\nr{
G
}_{L^{4/3} L^{4/3}}
\nr{
w_{\psi_1} - w_{\psi_2} 
}_{L^4 L^4}
\\
& \lesssim
\nr{
G
}_{\aY}
\nr{\psi_1 - \psi_2},
\quad
\psi_1,\psi_2\in \mh,
\end{align*} 
which is given by 
Proposition \ref{prop:St}.

Henceforth, 
we put $\psi\in \Sc \oplus \Sc$.
By (\ref{L1:prf:prop:K}),  
we obtain 
\begin{align*}
&
\bra{
\int_{\R}
\binom{
-\oi \sin(t\omega)
}{
 \cos(t\omega)   
}
G(t)
dt,
\Miijj{0}{-\omega^{-2}}{1}{0}\psi
}_{\mh}
\\
&=
\bra{
\int_{\R}
\bra{x}^{-1}\omega^{-1/2}
\binom{
-\sin(t\omega)  
}{
 \cos(t\omega)  
}
G(t)
dt
,
\bra{x}\omega^{1/2}
\Miijj{0}{-\omega^{-1}}{1}{0}\psi
}_{L^2 \oplus L^2}
\\
&=
\int_{\R}
\bra{
\bra{x}^{-1}\omega^{-1/2}
\binom{
-\sin(t\omega)  
}{
 \cos(t\omega) 
}
G(t)
,
\bra{x}\omega^{1/2}
\Miijj{0}{-\omega^{-1}}{1}{0}\psi
}_{L^2 \oplus L^2}
dt.
\end{align*} 
We see from 
the assumption $G\in C^0L^2$ that 
\begin{align*}
&
\bra{
\bra{x}^{-1}\omega^{-1/2}
\binom{
-\sin(t\omega)  
}{
 \cos(t\omega) 
}
G(t)
,
\bra{x}\omega^{1/2}
\Miijj{0}{-\omega^{-1}}{1}{0}\psi
}_{L^2 \oplus L^2}
\\
&=
\bra{
\binom{
-\sin(t\omega) 
}{
 \cos(t\omega)
}
G(t)
,
\Miijj{0}{-\omega^{-1}}{1}{0}\psi
}_{L^2 \oplus L^2}
\\
&=
\bra{
\Miijj{\cos(t\omega)}{-\sin(t\omega)
}{\sin(t\omega)}{\cos(t\omega)}
\binom{
0
}{
G(t)
}
,
\Miijj{0}{-\omega^{-1}}{1}{0}\psi
}_{L^2 \oplus L^2}
\\
&=
\bra{
\binom{
0
}{
G(t)
}
,
\Miijj{\cos(t\omega)}{\sin(t\omega)
}{-\sin(t\omega)}{\cos(t\omega)}
\Miijj{0}{-\omega^{-1}}{1}{0}\psi
}_{L^2 \oplus L^2}
\\
&=
\bra{ G(t), w_\psi(t) }_{L^2},
\quad
t\in \R.
\end{align*} 
Hence we have the desired identity.
\end{proof}

Expansion (\ref{SAL}) 
follows from (\ref{id:K}) and 
the following approximation of $\n(u_\phi)$:
\begin{align*}
\n(u_\phi)
=
\n(w_\phi)
+
O(\nr{\phi}^5)
\quad\text{as $\nr{\phi}\to 0$ in $\aY$.}
\end{align*}  
In order to obtain higher-order expansions,
we establish 
a revised approximation of $\n(u_\phi)$.
For this purpose, 
for $\phi\in \mh$ 
we define functions 
$A_0(\phi)=\n(w_\phi)$,
and 
\begin{align*}
A_n(\phi)
=
A_0(\phi)
+
\sum_{k=1}^n \frac{1}{k!} 
\n^{(k)}(w_\phi) 
\braa{ \Gs A_{n-k}(\phi) }^k
\end{align*}
if $n\in \N$ and 
$A_0(\phi),\cdots,A_{n-1}(\phi) \in \aY$.
We now show if $\nr{\phi}\ll 1$ 
then $A_n(\phi)$ is well-defined 
and 
we obtain a revised approximation of $\n(u_\phi)$.

\begin{prop}\label{prop:A:1}  
Fix $n\in \Z$.
Then there exists 
some $\delta_n \in (0,\eta_0]$ such that 
if $\phi \in B_{\delta_n}\mh$ then 
$A_n(\phi)$ is defined as a function in $\aY$, 
and  
\begin{align}
\nr{A_n(\phi)}_{\aY} 
&\lesssim_n
\nr{\phi}^3,
\label{est:1:prop:A:1}
\\
\nr{\n(u_\phi)-A_n(\phi)}_{\aY}
&\lesssim_n
\nr{\phi}^{2n+5}.
\label{est:2:prop:A:1}
\end{align}  
\end{prop}

\begin{proof}
We have only to show (\ref{est:2:prop:A:1}) 
since (\ref{est:1:prop:A:1}) follows from 
(\ref{est:2:prop:A:1}) and Estimate (\ref{est:u}). 
Let $\varepsilon_n$ ($n\in \Z$) be the positive number 
appearing in 
Proposition \ref{prop:N}.
By Proposition \ref{prop:St} and (\ref{est:u}),
we have for any $\phi\in B_{\eta_0}\mh$,
\begin{align}\label{est:1:prf:prop:A:1}
\nr{
w_\phi
}_{\wZ} \lesssim \nr{\phi}
\quad\text{and}\quad
\nr{
\Gs \n(u_\phi)
}_{\wZ}
\lesssim
\nr{
\n(u_\phi)
}_{\aY}
\lesssim
\nr{\phi}^3.
\end{align}
Therefore, 
for any $n\in \Z$, 
some $\varepsilon^\prime_n \in (0,\eta_0]$ satisfies 
\begin{align}\label{est:2:prf:prop:A:1}
\nr{
w_\phi + \theta \Gs \n(u_\phi)
}_{\li}
\le \varepsilon_n,
\quad
\phi \in B_{\varepsilon^\prime_n}\mh,
\ \theta\in [0,1].
\end{align}

Henceforth,
we prove (\ref{est:2:prop:A:1}) by mathematical induction.
Propositions \ref{prop:Taylor}, \ref{prop:N} 
and  
Estimates 
(\ref{est:1:prf:prop:A:1}), (\ref{est:2:prf:prop:A:1}) 
imply that for any $\phi\in \mh$ 
with 
$\nr{\phi}\le \varepsilon_1^\prime$,
\begin{align*}
\nr{\n(u_\phi)-A_0(\phi)}_{\aY}
&=
\nr{\n
\braa{
w_\phi +\Gs \n(u_\phi)
}
-\n(w_\phi)}_{\aY}
\\
&\le 
\int_0^1 
\nr{
\n^{(1)}\braa{
w_\phi + \theta \Gs \n(u_\phi)
}
\Gs \n(u_\phi)
}_{\aY}
d\theta
\\
&\lesssim
\int_0^1 
\nr{
w_\phi + \theta \Gs \n(u_\phi)
}_{\wZ}^2
\nr{
\n(u_\phi)
}_{\aY}
d\theta
\\
&\lesssim
\braa{
\nr{\phi}+\nr{\phi}^3
}^2
\nr{\phi}^3
\lesssim
\nr{\phi}^5. 
\end{align*} 
Hence we obtain
(\ref{est:2:prop:A:1}) with $n=0$ 
if $\delta_0=\varepsilon_1^\prime$ 
and $\phi \in B_{\delta_0}\mh$.
\\

We next assume that 
some $m\in \N$ and 
$\delta_0,\cdots,\delta_{m-1}\in (0,\eta_0]$
satisfy  
that 
if 
$n=0,\cdots,m-1$ and 
$\phi\in B_{\delta_n}\mh$
then 
(\ref{est:1:prop:A:1}) and (\ref{est:2:prop:A:1}) 
hold true.
Fix $\phi\in \mh$ 
with 
$
\nr{\phi}\le \delta_m
:=
\delta_0 \wedge \cdots \wedge \delta_{m-1} 
\wedge \varepsilon_{m+1}^\prime.
$
By Proposition \ref{prop:Taylor},
we have 
\begin{align*}
\n(u_\phi)
&=
\n\braa{
w_\phi +\Gs \n(u_\phi)
}
=
\n(w_\phi)
+
\sum_{k=1}^m \frac{1}{k!}
\n^{(k)}(w_\phi)
\braa{
\Gs \n(u_\phi)
}^k
\\
&\quad +
\int_0^1 \frac{(1-\theta)^m}{m!}
\n^{(m+1)}\braa{
w_\phi +\theta\Gs \n(u_\phi)
}
\braa{
\Gs \n(u_\phi)
}^{m+1} d\theta,
\end{align*} 
which implies that  
\begin{align*}
&
\n(u_\phi)-A_m(\phi)
\\
&=
\sum_{k=1}^m \frac{1}{k!}
\n^{(k)}(w_\phi)
Q_{k-1}\braa{
\Gs \n(u_\phi), \Gs A_{m-k}(\phi)
}
\Gs \braa{ 
\n(u_\phi)
-
A_{m-k}(\phi)
}
\\
&\quad +
\int_0^1 \frac{(1-\theta)^m}{m!}
\n^{(m+1)}\braa{
w_\phi +\theta\Gs \n(u_\phi)
}
\braa{
\Gs \n(u_\phi)
}^{m+1} d\theta.
\end{align*} 
Here, 
we have defined 
$Q_k(a,b)=\sum_{\ell=0}^k a^\ell b^{k-\ell}$ 
($k\in \Z$, $a,b \in \R$).
By 
Propositions \ref{prop:Taylor}--\ref{prop:N},
Estimates
(\ref{est:1:prf:prop:A:1}), 
(\ref{est:2:prf:prop:A:1}) 
and 
Assumptions  
(\ref{est:1:prop:A:1}), 
(\ref{est:2:prop:A:1}) with $n=0,\cdots,m-1$, 
we obtain 
\begin{align*}
&
\nr{
\n(u_\phi)-A_m(\phi)
}_{\aY}
\\
&\lesssim_m
\sum_{k=1}^m 
\nr{\phi}^{(3-k)_+} Q_{k-1}\braa{
\nr{\phi}^3, \nr{\phi}^3
}
\nr{\phi}^{2(m-k)+5}
+
\nr{\phi}^{(2-m)_+} \nr{\phi}^{3(m+1)}
\\
&\lesssim_m
\sum_{k=1}^m 
\nr{\phi}^{2m+5+(k-3)+(3-k)_+}
+
\nr{\phi}^{2m+5+(m-2)+(2-m)_+} 
\lesssim_m
\nr{\phi}^{2m+5}.
\end{align*} 
Hence we obtain
(\ref{est:2:prop:A:1}) with $n=m$ 
if $\phi \in B_{\delta_m}\mh$.
\end{proof}

We now give 
higher-order expansions of $\s$.
\begin{lem}\label{lem:A:1}  
Assume 
$\phi\in \mh$ and $n\in \Z$.
Then it follows that 
\begin{align*}
\abs{
K^{\phi,\lambda}
-
\int_{\R^{1+2}}
A_n(\lambda \phi)w_\phi
d(t,x)
}
\lesssim_{n,\nr{\phi}}
\lambda^{2n+5}
\end{align*} 
if $0<\lambda\ll 1$.
\end{lem}

\begin{proof}
We see from 
the H\"older inequality
and 
Propositions \ref{prop:A:1}, 
\ref{prop:St} 
that 
if $0<\lambda \ll 1$ then 
$A_n(\lambda \phi)w_\phi \in L^1(\R^{1+2})$
and 
\begin{align*}
\int_{\R^{1+2}}
\abs{
\braa{
\n(u_{\lambda \phi})-A_n(\lambda \phi)
}
w_\phi
}
d(t,x)
&\le 
\nr{
\n(u_{\lambda \phi})-A_n(\lambda \phi)
}_{L^{4/3}L^{4/3}}
\nr{w_\phi}_{L^4 L^4}
\\
&\lesssim_n
\lambda^{2n+5} \nr{\phi}^{2n+6}.
\end{align*} 
Hence the desired identity 
follows from (\ref{id:K}).
\end{proof}

We next differentiate the function $A_n(\lambda \phi)$
with respect to $\lambda$.

\begin{prop}\label{prop:A:2}  
Fix $\phi\in \mh$.
For any $n,N\in \Z$, 
there exists some $\Lambda_{n,N}\in (0,1)$
such that 
\begin{enumerate}[(D-i)]
  \item 
$
A_n(\lambda \phi) \in 
C^N((-\Lambda_{n,N},\Lambda_{n,N});\aY).
$
  \item 
If 
$\lambda\in (-\Lambda_{n,N},\Lambda_{n,N})$ then
\begin{align*}
\nr{
\Dl^N
A_n(\lambda \phi)
}_{\aY}
\lesssim_{n,N,\nr{\phi}} |\lambda|^{(3-N)_+}.
\end{align*} 
  \item 
If $3 \le N \le 2n+4$
and 
$\lambda\in (-\Lambda_{n,N},\Lambda_{n,N})$ then
\begin{align*}
\nr{
\Dl^N
A_n(\lambda \phi)
-
\widetilde{W}_N^\phi 
\brac{
\n^{(3)}(0),\cdots,\n^{(N)}(0)
}
}_{\aY}
\lesssim_{n,N,\nr{\phi}}
|\lambda|.
\end{align*} 
\end{enumerate}
\end{prop}

\begin{proof}
For $n\in \Z$, 
we say that Property $P(n)$ is true
if for any $N\in \Z$ 
some $\Lambda_{n,N}\in (0,1)$ satisfies 
(D-i)--(D-iii).

Henceforth, 
we prove the proposition by mathematical induction.
By
Propositions \ref{prop:Taylor}--\ref{prop:N} 
and \ref{prop:DN},  
for any $N\in \Z$ we have some $\Lambda_{0,N}\in (0,1)$ 
satisfying 
$A_0(\lambda \phi)\in C^N((-\Lambda_{0,N},\Lambda_{0,N}),\aY)$,
\begin{align*}
&
\nr{
\Dl^N A_0(\lambda \phi)
}_{\aY}
=
\nr{
\n^{(N)}(\lambda w_\phi)w_\phi^N
}_{\aY}
\\
&\lesssim_{N}
\braa{
\nr{\lambda w_\phi}_{\wZ}^{(3-N)_+} 
+
\nr{\lambda w_\phi}_{\wZ}^{1+(2-N)_+} 
}
\nr{w_\phi}_{\wZ}^N
\\
&\lesssim_{N}
|\lambda|^{(3-N)_+}
\braa{
\nr{\phi}^{(3-N)_+} 
+
\nr{\phi}^{1+(2-N)_+} 
}
\nr{\phi}^N,
\quad
\lambda\in (-\Lambda_{0,N},\Lambda_{0,N}),
\end{align*} 
and that if $3 \le N \le 4$ then 
\begin{align}
&
\nr{
\Dl^N
A_0(\lambda \phi)
-
\widetilde{W}_N^\phi 
\brac{
\n^{(3)}(0),\cdots,\n^{(N)}(0)
}
}_{\aY}
=
\nr{
\n^{(N)}(\lambda w_\phi)w_\phi^N
-
\n^{(N)}(0)w_\phi^N
}_{\aY}
\nonumber
\\
&\le
|\lambda|
\int_0^1
\nr{
\n^{(N+1)}(\theta \lambda w_\phi)w_\phi^{N+1}
}_{\aY}
d\theta
\nonumber
\\
&\lesssim_N
|\lambda|
\nr{\phi}^{N+1},
\quad
\lambda\in (-\Lambda_{0,N},\Lambda_{0,N}).
\label{est:prf:prop:A:2}
\end{align} 
Hence we have Property $P(0)$.
\\

We next assume that 
there exists some $m\in \N$ 
such that 
Property $P(n)$ is true
for any $n=0,\cdots,m-1$. 
Fix $N\in \Z$.
Set 
\begin{align*}
\Lambda_{m,N}
=
\Lambda_{0,N} \wedge \cdots \wedge \Lambda_{m-1,N} \wedge
\varepsilon_0 \wedge \cdots \wedge \varepsilon_{m+N+2}.
\end{align*} 
We see from  
$P(n)$ with $n=0,\cdots,m-1$ 
and the proof of $P(0)$
that  
if 
$\lambda \in (-\Lambda_{m,N},\Lambda_{m,N})$,
then 
\begin{align*}
&\Dl^N A_m(\lambda \phi)
=
\Dl^N A_0(\lambda \phi) 
\\
&\quad +
\sum_{k=1}^m \frac{1}{k!}
\sum_{N_0+N_1+\cdots+N_k=N} 
\frac{N!}{N_0!\cdots N_k!}
\n^{(k+N_0)}(\lambda w_\phi)w_\phi^{N_0}
\prod_{\ell=1}^k 
\Gs \Dl^{N_\ell} A_{m-k}(\lambda\phi)
\end{align*} 
and 
\begin{align*}
&
\nr{
\Dl^N
A_m(\lambda \phi)
}_{\aY}
\lesssim_{m,N}
\nr{
\Dl^N
A_0(\lambda \phi)
}_{\aY}
\\
&\quad +
\sum_{k=1}^m 
\sum_{N_0+N_1+\cdots+N_k=N} 
\nr{
\n^{(k+N_0)}(\lambda w_\phi)w_\phi^{N_0}
\prod_{\ell=1}^k 
\Gs \Dl^{N_\ell} A_{m-k}(\lambda\phi)
}_{\aY}
\\
&\lesssim_{m,N,\nr{\phi}}
|\lambda|^{(3-N)_+}
+
\sum_{k=1}^m 
\sum_{N_0+N_1+\cdots+N_k=N} 
|\lambda|^{(3-k-N_0)_+}
|\lambda|^{(3-N_1)_+}
\cdots
|\lambda|^{(3-N_k)_+}
\\
&\lesssim_{m,N,\nr{\phi}}
|\lambda|^{(3-N)_+}.
\end{align*} 
Henceforth,
we assume $3 \le N \le 2m+4$.
Then we obtain $m\ge k_N$,
and for any 
$\lambda \in (-\Lambda_{m,N},\Lambda_{m,N})$ 
\begin{align*}
&
\nr{
\Dl^N
A_m(\lambda \phi)
-
\widetilde{W}_N^\phi 
\brac{
\n^{(3)}(0),\cdots,\n^{(N)}(0)
}
}_{\aY}
\\
&\lesssim_{m,N}
\nr{
\Dl^N
A_0(\lambda \phi)
-
\n^{(N)}(0) w_\phi^N
}_{\aY}
+
\sum_{k=1}^{k_N}
\sum_{\substack{ 
N_0+N_1+\cdots+N_k=N
\\ 
N_0\ge (3-k)_+
\\
N_1,\cdots,N_k\ge 3
}}
b_{k,N_0,\cdots,N_k}(\lambda) 
\\
&\quad +
\sum_{k=1}^m
\sum_{\substack{ 
N_0+N_1+\cdots+N_k=N
\\ 
N_0 <(3-k)_+ \text{ or } \exists \ell \ N_\ell <3
}}
\nr{
\n^{(k+N_0)}(\lambda w_\phi)w_\phi^{N_0}
\prod_{\ell=1}^k 
\Gs \Dl^{N_\ell} A_{m-k}(\lambda\phi)
}_{\aY},
\end{align*} 
where we have defined 
for $k\in \N$ 
and $N_0,\cdots,N_k \in \Z$ 
with $N_0 + \cdots + N_k=N$,
$N_0\ge (3-k)_+$ and $N_1,\cdots,N_k\ge 3$,
\begin{align*}
&
b_{k,N_0,\cdots,N_k}(\lambda) 
=
\nr{
 \braa{
\n^{(k+N_0)}(\lambda w_\phi) w_\phi^{N_0}
-
\n^{(k+N_0)}(0)              w_\phi^{N_0}
 }
\prod_{\ell=1}^k 
\Gs \Dl^{N_\ell} A_{m-k}(\lambda\phi)
}_{\aY}
\\
&\quad +
\bbb{2.5}\|
\n^{(k+N_0)}(0)              w_\phi^{N_0}
\sum_{j=1}^k 
\bbb{2.5}\{
\Gs
 \braa{
\Dl^{N_j} A_{m-k}(\lambda\phi) 
-  
\widetilde{W}_{N_j}^\phi 
\brac{
\n^{(3)}(0),\cdots,\n^{(N_j)}(0)
}
 }
\\
&\qquad \times 
 \braa{
\prod_{\ell=1}^{j-1}
\widetilde{W}_{N_\ell}^\phi 
\brac{
\n^{(3)}(0),\cdots,\n^{(N_\ell)}(0)
}
 }
 \braa{
\prod_{\ell=j+1}^k
\Gs \Dl^{N_\ell} A_{m-k}(\lambda\phi)
 }
\bbb{2.5}\}
\bbb{2.5}\|_{\aY}.
\end{align*} 
Remark that 
as in the proof of 
(\ref{est:prf:prop:A:2}), 
we obtain 
\begin{align*}
\nr{
\Dl^N
A_0(\lambda \phi)
-
\n^{(N)}(0) w_\phi^N
}_{\aY}
\lesssim_{N,\nr{\phi}}
|\lambda|, 
\quad
\lambda \in (-\Lambda_{m,N},\Lambda_{m,N}).
\end{align*} 
For any $k\in \N$ 
and $N_0,\cdots,N_k \in \Z$ 
with $N_0 + \cdots + N_k=N$,
$N_0\ge (3-k)_+$ and $N_1,\cdots,N_k\ge 3$,
we have 
$3 \le N_j \le 2(m-k)+4$ 
($j=1,\cdots,k$), 
and 
it follows from 
Propositions \ref{prop:Taylor}--\ref{prop:N} 
and $P(n)$ with $n=0,\cdots,m-1$ 
that 
\begin{align*}
&
b_{k,N_0,\cdots,N_k}(\lambda)  
\lesssim_{m,N,\nr{\phi}}
|\lambda|
\int_0^1 
\nr{
\n^{(k+N_0+1)}(\theta\lambda w_\phi) w_\phi^{N_0+1}
\prod_{\ell=1}^k 
\Gs \Dl^{N_\ell} A_{m-k}(\lambda\phi)
}_{\aY}
d\theta
\\
&\quad +
\sum_{j=1}^k 
\nr{
\Dl^{N_j} A_{m-k}(\lambda\phi) 
-  
\widetilde{W}_{N_j}^\phi 
\brac{
\n^{(3)}(0),\cdots,\n^{(N_j)}(0)
}
}_{\aY}
\prod_{\ell=j+1}^k
\nr{
\Dl^{N_\ell} A_{m-k}(\lambda\phi)
}_{\aY}
\\
&\lesssim_{m,N,\nr{\phi}}
|\lambda|, 
\quad
\lambda \in (-\Lambda_{m,N},\Lambda_{m,N}),
\end{align*} 
and that 
\begin{align*}
&
\sum_{k=1}^m
\sum_{\substack{ 
N_0+N_1+\cdots+N_k=N
\\ 
N_0 <(3-k)_+ \text{ or } \exists \ell \ N_\ell <3
}}
\nr{
\n^{(k+N_0)}(\lambda w_\phi)w_\phi^{N_0}
\prod_{\ell=1}^k 
\Gs \Dl^{N_\ell} A_{m-k}(\lambda\phi)
}_{\aY}
\\
&\lesssim_{m,N,\nr{\phi}}
\sum_{k=1}^m
\sum_{\substack{ 
N_0+N_1+\cdots+N_k=N
\\ 
N_0 <(3-k)_+ \text{ or } \exists \ell \ N_\ell <3
}}
|\lambda|^{(3-k-N_0)_+}
|\lambda|^{(3-N_1)_+}
\cdots
|\lambda|^{(3-N_k)_+}
\\
&\lesssim_{m,N,\nr{\phi}}
|\lambda|, 
\quad
\lambda \in (-\Lambda_{m,N},\Lambda_{m,N}).
\end{align*} 
Hence we have $P(m)$,
which completes the proof.
\end{proof}

By Proposition \ref{prop:A:2} and 
the proof of Lemma \ref{lem:A:1},
we obtain the following property:
\begin{lem}\label{lem:A:2}  
Assume 
$n,N\in \Z$ and $\phi,\psi \in \mh$.
Then 
there exists some $\Lambda>0$ such that 
\begin{align*}
\int_{\R^{1+2}}
A_n(\lambda \phi)w_\psi d(t,x)
\in 
C^N((-\Lambda,\Lambda)).
\end{align*} 
Furthermore,
it follows that 
if $3 \le N \le 2n+4$, 
then for any $\lambda\in (-\Lambda,\Lambda)$,
\begin{align*}
&
\abs{
\frac{d^N}{d \lambda^N}
\int_{\R^{1+2}}
A_n(\lambda \phi)w_\psi d(t,x)
-
\int_{\R^{1+2}}
\widetilde{W}_N^\phi 
\brac{
\n^{(3)}(0),\cdots,\n^{(N)}(0)
}
w_\psi d(t,x)
}
\\
&\lesssim_{n,N,\nr{\phi},\nr{\psi}}
|\lambda|.
\end{align*}
\end{lem}

\section{Proof of theorems}
In order to show Theorems 
\ref{thm:main}, \ref{thm:main:2} 
and \ref{thm:main:3}, 
we first introduce two lemmas.

\begin{lem}[Proposition A.3 in \cite{Sasaki2012}]%
\label{lem:difference}  
Fix $\lambda_0 >0$ and $L\in \Z$.
Let $h\in C^{L+1}((0,\lambda_0))$.
If 
\begin{align*}
\max_{\ell=0,\cdots,L+1}
\sup_{\lambda \in (0,\lambda_0)}
\abs{
\frac{d^\ell}{d\lambda^\ell}h(\lambda)
}
<\infty,
\end{align*} 
then it follows that 
\begin{align*}
\abs{
\frac{d^L}{d\lambda^L}h(\lambda)
-
\lambda^{-L}\varDelta_\lambda^L h(\lambda)
}
\le C\lambda,
\quad
\lambda
\in \braa{
0,\frac{\lambda_0}{L+1}
}
\end{align*} 
for some positive constant $C$ 
independent of $\lambda$.
\end{lem}

\begin{lem}\label{lem:W}
For any $N\in \N$ with $N\ge 5$ 
and for any $\phi\in \mh$, 
there exist 
some $L\in \N$,
$v_1,\cdots,v_L \in C^0 L^2 \cap \aY$ 
and 
polynomials 
$P_1,\cdots,P_L$ 
of $(N-4)$ variables  
which satisfy the following:
\begin{itemize}
  \item 
$P_1,\cdots,P_L$
are independent of $\phi$.
  \item 
For any $m=1,\cdots,L$,
we obtain 
$\nr{v_m}_{\aY}\lesssim_{N} \nr{\phi}^N$.
  \item 
We have for any $y_3,\cdots,y_{N-2}\in \R$,
\begin{align*}
W_N^\phi[y_3,\cdots,y_{N-2}]
=
\sum_{m=1}^L P(y_3,\cdots,y_{N-2}) v_m.
\end{align*} 
\end{itemize}
\end{lem}

We immediately see that 
Lemma \ref{lem:W} follows from 
the proof of Proposition \ref{prop:N} 
and mathematical induction.

\begin{proof}[Proof of Theorem \ref{thm:main}]
We show only (2) since (1) can be proved more easily.
Assume $N\in \N$ with $N\ge 5$ and that 
we have already known the values of 
$\n^{(3)}(0),\cdots,\n^{(N-2)}(0)$. 
Choose $n\in \N$ and $\phi \in \mh$
so that $2n+4 \ge N$ 
and 
\begin{align*}
\int_{\R^{1+2}}
w_\phi^{N+1}
d(t,x)
\neq 0.
\end{align*} 
Let $\lambda>0$ be sufficiently small.
Using Lemma \ref{lem:A:1},
we have 
\begin{align*}
\abs{
K^{\phi,\lambda}
-
\int_{\R^{1+2}}
A_n(\lambda \phi)w_\phi
d(t,x)
}
\lesssim_{n,\nr{\phi}}
\lambda^{2n+5}.
\end{align*}
Hence we obtain
\begin{align*}
\abs{
\lambda^{-N}\varDelta_\lambda^N
K^{\phi,\lambda}
-
\lambda^{-N}\varDelta_\lambda^N
\int_{\R^{1+2}}
A_n(\lambda \phi)w_\phi
d(t,x)
}
\lesssim_{N,n,\nr{\phi}}
\lambda^{2n+5-N}
\lesssim_{N,n,\nr{\phi}}
\lambda.
\end{align*} 
We see from 
Lemmas \ref{lem:A:2} and \ref{lem:difference} that 
\begin{align*}
&
\left|
\lambda^{-N}\varDelta_\lambda^N K^{\phi,\lambda}
-
\n^{(N)}(0)
\int_{\R^{1+2}}
w_\phi^{N+1}
d(t,x)
\right.
\\
&\qquad
\left.
-
\int_{\R^{1+2}}
W_N^\phi
\brac{
\n^{(3)}(0),\cdots,\n^{(N-2)}(0)
}
w_\phi
d(t,x)
\right|
\\
& =
 \abs{
\lambda^{-N}\varDelta_\lambda^N K^{\phi,\lambda}
-
\int_{\R^{1+2}}
\widetilde{W}_N^\phi 
\brac{
\n^{(3)}(0),\cdots,\n^{(N)}(0)
}
w_\phi d(t,x)
 }
\\
& \le
\abs{
\lambda^{-N}\varDelta_\lambda^N K^{\phi,\lambda}
-
\lambda^{-N}\varDelta_\lambda^N
\int_{\R^{1+2}}
A_n(\lambda \phi)w_\phi
d(t,x)
}
\\
&\quad +
\abs{
\lambda^{-N}\varDelta_\lambda^N
\int_{\R^{1+2}}
A_n(\lambda \phi)w_\phi
d(t,x)
-
\frac{d^N}{d \lambda^N}
\int_{\R^{1+2}}
A_n(\lambda \phi)w_\phi
d(t,x)
}
\\
&\quad +
\abs{
\frac{d^N}{d \lambda^N}
\int_{\R^{1+2}}
A_n(\lambda \phi)w_\phi
d(t,x)
-
\int_{\R^{1+2}}
\widetilde{W}_N^\phi 
\brac{
\n^{(3)}(0),\cdots,\n^{(N)}(0)
}
w_\phi d(t,x)
}
\\
& \lesssim_{N,n,\nr{\phi}} \lambda,
\end{align*}  
which completes the proof of (2).
\end{proof}

\begin{proof}[Proof of Theorem \ref{thm:main:2}]
We prove (\ref{est:thm:main:2})
by induction.
If $N=3,4$ and $\Lambda_N^\prime>0$ is sufficiently small, 
then (\ref{est:thm:main:2}) follows from 
Theorem \ref{thm:main}(1).

Henceforth,
we assume that 
some $N\ge 4$ and $\Lambda_N^\prime>0$
satisfy (\ref{est:thm:main:2}).
We abbreviate 
$
K^\lambda
=
\lambda^{-N-1} \varDelta_\lambda^{N+1}
K^{\phi[N+1],\lambda},
$
$w=w_{\phi_{N+1}}$ 
and 
$
W[y_3,\cdots,y_{N-1}]
=
W^{\phi[N+1]}_{N+1}[y_3,\cdots,y_{N-1}]
$
($y_3,\cdots,y_{N-1} \in \R$).
Then 
we have for any $\lambda\in (0,1]$,
\begin{align*}
&
\abs{
\n^{(N+1)}(0)-I^\lambda_{N+1}
}
\\
&=
\abs{
\n^{(N+1)}(0)
-
\frac{
\displaystyle{
K^\lambda
-
\int_{\R^{1+2}}
W
\brac{
I^{\lambda}_3,\cdots,I^{\lambda}_{N-1}
}
w
d(t,x)
}
}{
\displaystyle{
\int_{\R^{1+2}}
w^{N+2}
d(t,x)
}
}
}
\\
&\le 
\abs{
\n^{(N+1)}(0)
-
\frac{
\displaystyle{
K^\lambda
-
\int_{\R^{1+2}}
W
\brac{
\n^{(3)}(0),\cdots,\n^{(N-1)}(0)
}
w
d(t,x)
}
}{
\displaystyle{
\int_{\R^{1+2}}
w^{N+2}
d(t,x)
}
}
}
\\
&\quad +
\abs{
\frac{
\displaystyle{
\int_{\R^{1+2}}
W
\brac{
\n^{(3)}(0),\cdots,\n^{(N-1)}(0)
}
w
d(t,x)
-
\int_{\R^{1+2}}
W
\brac{
I^{\lambda}_3,\cdots,I^{\lambda}_{N-1}
}
w
d(t,x)
}
}{
\displaystyle{
\int_{\R^{1+2}}
w^{N+2}
d(t,x)
}
}
}
\\
&=:a_1+a_2.
\end{align*} 
We see from 
Theorem \ref{thm:main}(2) that 
$a_1 \lesssim_{N+1,\phi_{N+1}}\lambda$ 
if $\lambda\in (0,\Lambda_{N+1})$.
By Lemma \ref{lem:W},
we obtain for some 
$\Lambda_{N+1}^{\prime\prime}>0$,
\begin{align*}
a_2
&\lesssim_{N+1,\phi_{N+1}} 
\abs{
\int_{\R^{1+2}}
w
\braa{
W
\brac{
\n^{(3)}(0),\cdots,\n^{(N-1)}(0)
}
-
W
\brac{
I^{\lambda}_3,\cdots,I^{\lambda}_{N-1}
}
}
d(t,x)
}
\\
&\lesssim_{N+1,\phi_3,\cdots,\phi_{N+1}}
\sum_{n=3}^{N+1}
\abs{
\n^{(n)}(0)
-
I^\lambda_n
}
\\
&\lesssim_{N+1,\phi_3,\cdots,\phi_{N+1}}
\lambda,
\quad
\lambda \in (0,\Lambda_{N+1}^{\prime\prime}).
\end{align*} 
Therefore, 
it follows that 
for any $n=3,\cdots,N+1$,
\begin{align*}
\abs{
\n^{(n)}(0)
-
I^\lambda_n
}
\lesssim_{N,\phi_3,\cdots,\phi_{N+1}} \lambda,
\quad
\lambda \in 
(0,
\Lambda_N^\prime \wedge
\Lambda_{N+1} \wedge \Lambda_{N+1}^{\prime\prime}),
\end{align*} 
which completes the proof.
\end{proof}

\begin{proof}[Proof of Theorem \ref{thm:main:3}]
Fix $N\in \N$.
By the standard argument, 
we have an open set $U\subset \mh$ 
containing 0 
such that $\s$ is $(N+2)$-th order 
G\^{a}teaux differentiable  
as a mapping from $U$ to $\mh$.

Put $\phi\in \mh$ and $\psi \in \Sc \oplus \Sc$ 
arbitrarily.
Since $\Sc\oplus\Sc$ is dense in $\mh$, 
we have only to show  
\begin{align}\label{id:1:prf:thm:main:3}
\bra{
d^1 \s(0;\phi)
,
\psi
}_{\mh}
=
\bra{
\phi,
\psi
}_{\mh}
\quad\text{and}\quad
\bra{
d^2 \s(0;\phi)
,
\psi
}_{\mh}
=
0
\end{align} 
and 
\begin{align}\label{id:2:prf:thm:main:3}
\bra{
d^N \s(0;\phi)
,
\psi
}_{\mh}
=
\bra{
\int_{\R}
\binom{
-\oi\sin(t\omega) 
}{
 \cos(t\omega)   
}
\wW_N^\phi[\n^{(3)}(0),\cdots,\n^{(N)}(0)]
dt
,
\psi
}_{\mh}
\end{align} 
if $N\ge 3$.

The G\^{a}teaux differentiability mentioned above 
and 
Lemma \ref{lem:difference} 
imply that 
\begin{align}\label{id:3:prf:thm:main:3}
\bra{
d^N \s(0;\phi)
,
\psi
}_{\mh}
=
\lim_{\lambda\to 0}\frac{d^N}{d\lambda^N}
\bra{
\s(\lambda \phi)
,
\psi}_{\mh}
=
\lim_{\lambda\to +0}
\lambda^{-N}\varDelta_\lambda^N
\bra{
\s(\lambda \phi)
,
\psi}_{\mh}.
\end{align} 
Assume $0< \lambda \ll 1$ 
and 
let $\widetilde\psi \in \Sc \oplus \Sc$ 
and $n\in \Z$ 
satisfy 
\begin{align*}
\psi
=
\Miijj{0}{-\omega^{-2}}{1}{0}
\widetilde\psi
\quad\text{and}\quad
2n+4\ge N.
\end{align*} 
By (\ref{id:S}) 
and the proof of (\ref{id:K}), 
we obtain 
\begin{align*}
\bra{
\s(\lambda \phi)
,
\psi}_{\mh}
&=
\bra{
\lambda \phi
+
\int_{\R}
\binom{
-\oi\sin(t\omega) 
}{
 \cos(t\omega)   
}
\n(u_{\lambda \phi}(t))
dt
,
\psi}_{\mh}
\\
&=
\lambda 
\bra{
\phi
,
\psi}_{\mh}
+
\int_{\R^{1+2}} 
\n(u_{\lambda \phi})
w_{\widetilde{\psi}}d(t,x)
\\
&=
\lambda 
\bra{
\phi
,
\psi}_{\mh}
+
\int_{\R^{1+2}} 
A_n(\lambda \phi)
w_{\widetilde{\psi}}d(t,x)
\\
&\quad +
\int_{\R^{1+2}} 
\braa{
\n(u_{\lambda \phi})
-
A_n(\lambda \phi)
}
w_{\widetilde{\psi}}d(t,x).
\end{align*} 
It follows from 
the H\"older inequality,
(\ref{est:u})  
and Propositions 
\ref{prop:St}, \ref{prop:A:1}
that 
\begin{align*}
\abs{
\int_{\R^{1+2}} 
\n(u_{\lambda \phi})
w_{\widetilde{\psi}}d(t,x)
}
&\le 
\nr{
\n(u_{\lambda \phi})
}_{L^{4/3}L^{4/3}} 
\nr{
w_{\widetilde\psi}
}_{L^4 L^4}
\lesssim_{
\nr{\phi},\nr{\psi}
}
\lambda^3
\end{align*} 
and 
\begin{align*}
\abs{
\int_{\R^{1+2}} 
\braa{
\n(u_{\lambda \phi})
-
A_n(\lambda \phi)
}
w_{\widetilde{\psi}}d(t,x)
}
&\le 
\nr{
\n(u_{\lambda \phi})
-
A_n(\lambda \phi)
}_{L^{4/3}L^{4/3}} 
\nr{
w_{\widetilde\psi}
}_{L^4 L^4}
\\
&\lesssim_{
n,\nr{\phi},\nr{\psi}
}
\lambda^{2n+5},
\end{align*} 
which implies that 
\begin{align*}
&
\lim_{\lambda\to +0}
\lambda^{-1}\varDelta_\lambda^1
\bra{
\s(\lambda \phi)
,
\psi}_{\mh}
=
\bra{\phi,\psi}_{\mh}
\quad\text{and}\quad
\lim_{\lambda\to +0}
\lambda^{-2}\varDelta_\lambda^2
\bra{
\s(\lambda \phi)
,
\psi}_{\mh}
=
0,
\\
&
\lim_{\lambda\to +0}
\lambda^{-N}\varDelta_\lambda^N
\bra{
\s(\lambda \phi)
,
\psi}_{\mh}
=
\lim_{\lambda\to +0}
\lambda^{-N}\varDelta_\lambda^N
\int_{\R^{1+2}} 
A_n(\lambda \phi)
w_{\widetilde{\psi}}d(t,x),
\quad
N\ge 3.
\end{align*} 
Using (\ref{id:3:prf:thm:main:3}) 
and  
Lemmas \ref{lem:A:2}, 
\ref{lem:difference},
we have (\ref{id:1:prf:thm:main:3})
and 
\begin{align*}
\bra{
d^N \s(0;\phi)
,
\psi
}_{\mh}
&=
\lim_{\lambda\to 0}\frac{d^N}{d\lambda^N}
\int_{\R^{1+2}} 
A_n(\lambda \phi)
w_{\widetilde{\psi}}d(t,x)
\\
&=
\int_{\R^{1+2}}
\widetilde{W}_N^\phi 
\brac{
\n^{(3)}(0),\cdots,\n^{(N)}(0)
}
w_{\widetilde{\psi}}d(t,x),
\quad
N\ge 3.
\end{align*} 
Henceforth we assume $N\ge 3$.
We see from 
Remark \ref{rem:def W}(1) and 
Proposition \ref{prop:K} that 
\begin{align*}
\int_{\R^{1+2}}
&
\widetilde{W}_N^\phi 
\brac{
\n^{(3)}(0),\cdots,\n^{(N)}(0)
}
w_{\widetilde{\psi}}d(t,x)
\\
&=
\bra{
\int_{\R}
\binom{
-\oi\sin(t\omega) 
}{
 \cos(t\omega)   
}
\wW_N^\phi[\n^{(3)}(0),\cdots,\n^{(N)}(0)]
dt
,
\psi
}_{\mh}.
\end{align*} 
Hence (\ref{id:2:prf:thm:main:3})  holds true.
\end{proof}

\section*{Acknowledgments}

This work was supported by 
JSPS KAKENHI Grant Number 22K03367.
The author would like to thank the reviewers 
for their valuable comments and suggestions.

\appendix
\section{Proof of (\ref{id:rem:1:thm:main:3})}
Set $\phi\in \mh$ and $N\in \N$ with $N\ge 3$.
By Theorem \ref{thm:main:3}, 
we have only to show that 
\begin{align}\label{id:1:prf:rem:1:thm:main:3}
\wW_N^\phi[a,0,\cdots,0]
=
\left\{
  \begin{array}{cl}
\displaystyle{
a^{(N-1)/2}
\mathcal{W}_N^\phi
}
&\text{if $N$ is odd and $N\ge 3$,}\\
0
&\text{if $N$ is even and $N\ge 4$.}
  \end{array}
\right.
\end{align} 
It is clear that the above identity with $N=3,4$ is true.
It follows from (\ref{def:W}) that 
\begin{align}\label{id:2:prf:rem:1:thm:main:3}
&
\wW^\phi_N[a,0,\cdots,0]
\nonumber
\\
&=
\sum_{k=1}^{3 \wedge k_N}
\frac{N!}{k!}
\sum_{
 \substack{ 
 N_1+\cdots+N_k=N-3+k
 \\
 N_1,\cdots,N_k\ge 3
 }
}
\frac{a w_\phi^{3-k} }{(3-k)!}
\prod_{\ell=1}^k
\frac{
\Gs
\widetilde{W}^\phi_{N_\ell}[a,0,\cdots,0]
}{ N_\ell ! },
\quad
N\ge 5.
\end{align}
We see from 
(\ref{id:2:prf:rem:1:thm:main:3}) and 
mathematical induction that 
(\ref{id:1:prf:rem:1:thm:main:3})
is true when $N$ is even.
Henceforth, 
we put $n\in \N$.
Then we have 
$3 \wedge k_{2n+3}=n \wedge 3$ and 
\begin{align*}
\wW^\phi_{2n+3}[a,0,\cdots,0]
&=
\sum_{k=1}^{n \wedge 3}
\frac{(2n+3)!}{k!}
\sum_{
 \substack{ 
 N_1+\cdots+N_k=2n+k
 \\
 N_1,\cdots,N_k\ge 3
 \\
 N_1,\cdots,N_k \in \mathbb{O}
 }
}
\frac{a w_\phi^{3-k} }{(3-k)!}
\prod_{\ell=1}^k
\frac{
\Gs
\widetilde{W}^\phi_{N_\ell}[a,0,\cdots,0]
}{ N_\ell ! }
\\
&=
\sum_{k=1}^{n \wedge 3}
\frac{(2n+3)!}{k!(3-k)!}
\sum_{
 n_1+\cdots+n_k=n-k
}
a w_\phi^{3-k}
\prod_{\ell=1}^k
\frac{
\Gs
\widetilde{W}^\phi_{2n_\ell +3}[a,0,\cdots,0]
}{ (2n_\ell +3) ! }
,
\end{align*}
where $\mathbb{O}$ is the set of all odd integers.
If 
$
\wW^\phi_{2m+3}[a,0,\cdots,0]
=a^{m+1}\mathcal{W}_{2m+3}^\phi
$
for any $m\in \Z$ with $m<n$,
then 
we obtain
\begin{align*}
\wW^\phi_{2n+3}[a,0,\cdots,0]
&=
\sum_{k=1}^{n \wedge 3}
\frac{(2n+3)!}{k!(3-k)!}
\sum_{
 n_1+\cdots+n_k=n-k
}
a w_\phi^{3-k}
\prod_{\ell=1}^k
\frac{
a^{n_\ell+1}
\Gs
\mathcal{W}_{2n_\ell +3}^\phi
}{ (2n_\ell +3) ! }
\\
&=
a^{n+1}
\sum_{k=1}^{n \wedge 3}
\frac{(2n+3)!}{k!(3-k)!}
\sum_{
 n_1+\cdots+n_k=n-k
}
w_\phi^{3-k}
\prod_{\ell=1}^k
\frac{
\Gs
\mathcal{W}_{2n_\ell +3}^\phi
}{ (2n_\ell +3) ! }
\\
&=
a^{n+1} 
\mathcal{W}_{2n +3}^\phi.
\end{align*} 
It follows from mathematical induction 
that (\ref{id:1:prf:rem:1:thm:main:3}) holds true 
for any $N$.

\section{A function satisfying (\ref{nonzero})}
Let us show the existence of the function $\phi$ 
mentioned in Remark \ref{rem:1:thm:main:1}.
We have only to consider the case $N\ge 4$.
For $a\ge 0$, 
we set $\omega_a=\sqrt{a^2-\Delta}$.
Put $h_1(r)=e^{-r^2/2}$,
$h_2(x)=h_1(|x|)$, 
$f_a(x)=\omega_a^{-1} h_2(x)$ and 
$f_{a,m}(x)=f_a(x/m)$
($r\in \R$, $x\in \R^2$, $a\ge 0$, $m>0$). 
In three steps, 
we prove that 
if $\sigma$ and  $m$ are sufficiently small positive numbers,
then $\phi:={}^t(f_{\sigma,m},0)$ satisfies (\ref{nonzero}).

\subsection*{Step 1.}
Let $\F$ be the Fourier transform on $\Scp(\R)$ 
which satisfies 
\begin{align*}
(\F f)(\rho)=(2\pi)^{-1/2}\int_{\R} e^{i\rho  r}f(r)dr,
\quad
f\in L^1(\R), \ \rho \in \R.
\end{align*}  
By using the Bessel function  $J_0$ of the first kind,
we have for any $t\in \R$ and $x\in \R^2\setminus\{ 0 \}$, 
\begin{align*}
\cos(t\omega_0)f_0(x)
&=
\frac{1}{2\pi}
\int_{\R^2}
e^{ix\cdot \xi}\cos(t|\xi|)|\xi|^{-1}h_2(\xi)
d\xi
\\
&=
\int_0^\infty
\cos(tr)h_1(r)J_0(r|x|)
dr
\\
&=
\frac{\sqrt{2\pi}}{2}(\F(h_1 J_0(\,\cdot\, |x|)))(t)
\\
&=\frac{1}{2}\int_{\R}
(\F J_0(\,\cdot\, |x|))(t-\rho) (\F h_1)(\rho)
d\rho.
\end{align*}
Moreover, 
since $(\F h_1)(\rho)             =e^{-\rho^2/2}$ and 
\begin{align*}
(\F J_0(\,\cdot\, |x|))(\rho) =
\left\{
  \begin{array}{ll}
\sqrt{2/\pi}(|x|^2-\rho^2)^{-1/2}&(-|x|<\rho <|x|),         \\
0                             &(\text{otherwise}),   \\
  \end{array}
\right.
\end{align*} 
we see that 
\begin{align}\label{positivity}
\cos(t\omega_0)f_0(x)>0,
\quad
t\in \R, \ x\in \R^2.
\end{align}

\subsection*{Step 2.}
For $s\in \R$ and $p\in [1,\infty]$,
$\dot{H}^s_p$ (resp. $H^s_p$) denotes 
the homogeneous (resp. inhomogeneous) Sobolev space on $\R^2$.
Recall the $L^p$-$L^q$ estimate
\begin{align}\label{est:LpLq-wave}
\nr{e^{it\omega_0}f}_{p} \lesssim_p 
|t|^{-(1/2-1/p)} \nr{f}_{\dot{H}^{3(1/2-1/p)}_{p^\prime}},
\quad p\in [2,\infty), \ t\neq 0, \ 
f\in \dot{H}^{3(1/2-1/p)}_{p^\prime}
\end{align} 
for the wave equation,  
and
\begin{align}\label{est:LpLq-KG}
\nr{e^{it\omega}f}_{p} \lesssim_p 
|t|^{-(1/2-1/p)} \nr{f}_{H^{3(1/2-1/p)}_{p^\prime}},
\quad p\in [2,\infty), \ t\neq 0, \ 
f\in H^{3(1/2-1/p)}_{p^\prime}
\end{align} 
for the Klein-Gordon equation 
(see, e.g., \cite{Brenner1985} and \cite{Pecher1976}, respectively).
By (\ref{est:LpLq-wave}), 
embeddings  
$\dot{H}_2^{1-4/(3N)} \hookrightarrow  L^{3N/2}$ 
and 
$L^{6N/(3N+8)}        \hookrightarrow  \dot{H}^{-4/(3N)}_{2}$,
we have for any $t\in \R\setminus \{ 0\}$ and $\sigma>0$, 
\begin{align*}
\nr{ \cos(t\omega_0)f_0 }_{3N/2}
&\lesssim_N
\braa{
|t|^{-(1/2-2/(3N))}
\nr{f_0}_{\dot{H}^{3(1/2-2/(3N))}_{(3N/2)^\prime}}
}
\wedge 
\nr{\cos(t\omega_0)f_0}_{\dot{H}^{1-4/(3N)}_{2}}
\\
&\lesssim_N
\braa{
|t|^{-(1/2-2/(3N))}
\nr{h_2}_{\dot{H}^{1/2-2/N}_{(3N/2)^\prime}}
}
\wedge
\nr{h_2}_{6N/(3N+8)}
\end{align*} 
and 
\begin{align*}
\nr{ \cos(t\omega_0)f_\sigma }_{3N/2}
&\lesssim_N
\braa{
|t|^{-(1/2-2/(3N))}
\nr{\frac{\omega_0}{\omega_\sigma} h_2}_{\dot{H}^{1/2-2/N}_{(3N/2)^\prime}}
}
\wedge
\nr{\frac{\omega_0}{\omega_\sigma} h_2}_{6N/(3N+8)}.
\end{align*} 
The Mihlin theorem (see, e.g., \cite{BL}) implies that 
if $p\in (1,\infty)$ then 
$\omega_0/\omega_\sigma$ is a bounded linear operator 
from $L^p$ to $L^p$, 
and that 
$\sup_{\sigma>0}\nr{\omega_0/\omega_\sigma}_{L^p \to L^p}<\infty$. 
Hence we obtain for any $t\in \R$ and $a\ge 0$, 
\begin{align}\label{est:fa}
\nr{ \cos(t\omega_0)f_a }_{3N/2}
\lesssim_N
(1+|t|)^{-(1/2-2/(3N))} 
\nr{h_2}_{\dot{H}^{1/2-2/N}_{(3N/2)^\prime} \cap L^{6N/(3N+8)}}.
\end{align} 
We see from the Hausdorff-Young inequality and H\"older inequality that 
for any $t\in \R$ and $\sigma> 0$,
\begin{align}
\nr{
\cos(t\omega_0)f_\sigma - \cos(t\omega_0)f_0
}_3
&\lesssim
\nr{
\cos(t|x|)h_2(x) 
\braa{ \frac{1}{|x|}-\frac{1}{\sqrt{\sigma^2+|x|^2}}  }
}_{3/2}
\nonumber
\\
&\le
\nr{
\bra{x}^{-1}
\braa{ \frac{1}{|x|}-\frac{1}{\sqrt{\sigma^2+|x|^2}}  }
}_{12/7}
\nr{
\bra{x}
h_2(x) 
}_{12}. \label{est:HY}
\end{align} 
It follows from the Lebesgue dominated convergence theorem that 
\begin{align}\label{limit:fa}
\lim_{\sigma\to +0}
\nr{
\cos(t\omega_0)f_\sigma - \cos(t\omega_0)f_0
}_3
=0,
\quad t\in \R.
\end{align} 
Using (\ref{est:fa}),  
(\ref{est:HY}) and the H\"older inequality,
we have 
\begin{align*}
\int_{\R^{1+2}}
\abs{
\cos(t\omega_0)f_a
}^{N+1}
d(t,x)
&\le
\int_{\R}
\nr{ \cos(t\omega_0)f_a }_3 
\nr{ \cos(t\omega_0)f_a }_{3N/2}^{N}
dt
\\
&\lesssim_N
\int_{\R}
(1+|t|)^{-N(1/2-2/(3N))}
dt
<\infty,
\quad
a\ge 0.
\end{align*} 
Moreover,
(\ref{limit:fa}) and 
the Lebesgue dominated convergence theorem
imply that 
\begin{align}
&
\lim_{\sigma\to +0}
\abs{
\int_{\R^{1+2}}
\braa{
\cos(t\omega_0)f_\sigma
}^{N+1}
d(t,x)
-
\int_{\R^{1+2}}
\braa{
\cos(t\omega_0)f_0
}^{N+1}
d(t,x)
}
\nonumber
\\
&\lesssim_N
\liminf_{\sigma\to +0}
\int_{\R}
\nr{
\cos(t\omega_0)f_\sigma - \cos(t\omega_0)f_0
}_3
(1+|t|)^{-N(1/2-2/(3N))}
dt=0. \label{limit:fa-f0}
\end{align} 
By Step 1,
we obtain 
$
\int_{\R^{1+2}}
\braa{
\cos(t\omega_0)f_0
}^{N+1}
d(t,x)>0.
$ 
Therefore,
there exists some $\sigma_1>0$ such that 
$
\int_{\R^{1+2}}
\braa{
\cos(t\omega_0)f_{\sigma}
}^{N+1}
d(t,x)>0
$
for any $\sigma\in (0,\sigma_1)$.

\subsection*{Step 3.}
Put $g\in \Sc$ and $m>0$. 
Define $g_m(x)=g(x/m)$ ($x\in \R^2$).
By direct calculation,
we have 
\begin{align*}
(\cos(t\omega)g_m)(x)
=
\braa{
\cos(t\sqrt{1-\Delta/m^2})g
}(x/m)
=
\braa{
\cos((t/m)\omega_m)g
}(x/m)
\end{align*} 
for any $t\in \R$ and $x\in \R^2$, 
and 
\begin{align}\label{id:cos-gm}
\int_{\R^{1+2}} 
(\cos(t\omega)g_m)^{N+1}
d(t,x)
=
m^3
\int_{\R^{1+2}} 
(\cos(t\omega_m)g)^{N+1}
d(t,x).
\end{align} 
By (\ref{est:LpLq-KG}) and the embedding 
$H^{1-1/N} \hookrightarrow L^{2N}$, 
we obtain for any $t\in \R \setminus \{ 0 \}$,
\begin{align*}
&
\nr{ \cos(t\omega_m)g }_{2N}
=
\nr{ \braa{ \cos(mt \omega)g\braa{\frac{\cdot}{m}} } (mx) }_{2N}
=
m^{-2/(2N)}
\nr{ \cos(mt \omega)g\braa{\frac{\cdot}{m}} }_{2N}
\\
&\lesssim_N
m^{-1/N} |mt|^{-(1/2-1/(2N))}
\nr{ g\braa{\frac{\cdot}{m}} }_{H^{3(1/2-1/(2N))}_{(2N)^\prime}}
\\
&\lesssim_N
m^{-1/N} |mt|^{-(1/2-1/(2N))} m^{-3(1/2-1/(2N))} m^{2/(2N)^\prime}
\nr{ g }_{H^{3(1/2-1/(2N))}_{(2N)^\prime}}
\\
&=
|t|^{-(1/2-1/(2N))} 
\nr{ g }_{H^{3(1/2-1/(2N))}_{(2N)^\prime}}
\end{align*} 
and 
\begin{align*}
\nr{ \cos(t\omega_m)g }_{2N}
\lesssim_N
\nr{ \cos(t\omega_m)g }_{H^{1-1/N}}
\le 
\nr{ g }_{H^{1-1/N}},
\end{align*} 
which implies that for any $t\in \R$, 
\begin{align}\label{est:g:2N}
\nr{ \cos(t\omega_m)g }_{2N}
\lesssim_N
(1+|t|)^{-(1/2-1/(2N))} 
\nr{ g }_{H^{3(1/2-1/(2N))}_{(2N)^\prime} \cap H^{1-1/N}}.
\end{align} 
We see from (\ref{est:LpLq-wave}) 
that (\ref{est:g:2N}) still holds true 
even if $m=0$.
It follows from the Parseval inequality 
and 
the Lebesgue dominated convergence theorem that 
\begin{align*}
\lim_{m\to +0}
\nr{
\cos(t\omega_m)g - \cos(t\omega_0)g
}_2
=0,
\quad t\in \R.
\end{align*} 
As in the proof of (\ref{limit:fa-f0}), 
we have 
\begin{align*}
\lim_{m\to +0}
\int_{\R^{1+2}}
\braa{
\cos(t\omega_m)g
}^{N+1}
d(t,x)
=
\int_{\R^{1+2}}
\braa{
\cos(t\omega_0)g
}^{N+1}
d(t,x).
\end{align*} 
Fix $\sigma\in (0,\sigma_1)$.
Since $f_\sigma \in \Sc$ and 
$
\int_{\R^{1+2}}
\braa{
\cos(t\omega_0)f_{\sigma}
}^{N+1}
d(t,x)>0,
$
there exists some $m(\sigma)>0$ such that 
\begin{align*}
\int_{\R^{1+2}}
\braa{
\cos(t\omega)f_{\sigma,m}
}^{N+1}
d(t,x)
=
m^3
\int_{\R^{1+2}}
\braa{
\cos(t\omega_m)f_{\sigma}
}^{N+1}
d(t,x)
>0
\end{align*} 
for any $m\in (0,m(\sigma))$.
Therefore, 
if $\sigma\in (0,\sigma_1)$ and $m\in (0,m(\sigma))$,
then $\phi:={}^t(f_{\sigma,m},0)$ satisfies (\ref{nonzero}).


\begin{thebibliography}{10}
\bibitem{BarretoUhlmannWang2022}
A.S. Barreto, G. Uhlmann, and Y. Wang, 
Inverse scattering for critical semilinear wave equations, 
Pure Appl. Anal. 4 (2022), 
191--223.



\bibitem{BL}
J. Bergh and J. L\"ofstr\"om, 
Interpolation spaces. An introduction,  
Grundlehren der Mathematischen Wissenschaften, 
No. 223. Springer-Verlag, Berlin-New York, 1976.



\bibitem{Brenner1985}
P. Brenner,
On scattering and everywhere defined scattering operators 
for nonlinear Klein-Gordon equations,
J. Differential Equations 56 (1985) 310--344. 



\bibitem{ChenMurphy-pre}
G. Chen and J. Murphy,
Recovery of the nonlinearity from the modified scattering map,
Int. Math. Res. Not. 2024 (2024) 6632--6655.



\bibitem{ChenMurphy-pre-2}
G. Chen and J. Murphy,
Stability estimates for the recovery of the nonlinearity from scattering data,
Pure Appl. Anal. 6 (2024) 305--317. 



\bibitem{Gateaux1919}
R. G\^{a}teaux,
Fonctions d'une infinite de variables indep\'{e}ndantes. 
Bull. Soc. Math. France 47 (1919) 70--96. 



\bibitem{HillePhillips1974}
E. Hille and R.S. Phillips, 
Functional Analysis and Semi-groups,
Third printing of the revised edition of 1957. 
American Mathematical Society Colloquium Publications, 
Vol. XXXI American Mathematical Society, 
Providence, R.I., 1974.



\bibitem{HoganMurphyGrow2023}
C.C. Hogan, J. Murphy and D. Grow, 
Recovery of a cubic nonlinearity for the nonlinear Schrodinger equation,
J. Math. Anal. Appl. 522 (2023),
127016.



\bibitem{HuKillipVisan-pre}
N. Hu, R. Killip and M. Visan,
Deconvolutional determination of the nonlinearity in a semilinear wave equation,
Pure Appl. Anal. 7 (2025) 1--17. 



\bibitem{IbrahimMajdoubMasmoudi2006}
S. Ibrahim, M. Majdoub, and N. Masmoudi, 
Global solutions for a semilinear, 
two-dimensional Klein-Gordon equation 
with exponential-type nonlinearity, 
Comm. Pure Appl. Math. 59 (2006), 1639--1658.



\bibitem{IbrahimMajdoubMasmoudiNakanishi2009}
S. Ibrahim, M. Majdoub, N. Masmoudi, and K. Nakanishi, 
Scattering for the two-dimensional energy-critical wave equation, 
Duke Math. J. 150 (2009), 287--329.



\bibitem{IbrahimMasmoudiNakanishi2009}
S. Ibrahim, N. Masmoudi, and K. Nakanishi, 
Scattering threshold for the focusing nonlinear
Klein-Gordon equation, 
Anal. PDE 4 (2011), 405--460. 
Correction to the article Scattering threshold 
for the focusing nonlinear Klein-Gordon equation, 
Anal. PDE 9 (2016), 503--514.



\bibitem{IkedaInuiOkamoto2020}
M. Ikeda and T. Inui,
Scattering for the one-dimensional Klein-Gordon equation 
with exponential nonlinearity,
J. Hyperbolic Differ. Equations 17 (2020),
295--354.



\bibitem{KillipMurphyVisan2023}
R. Killip, J. Murphy and M. Visan,
The scattering map determines the nonlinearity,
Proc. Amer. Math. Soc. 151 (2023), 
2543--2557.



\bibitem{MSSSS2018}
M. Maeda, H. Sasaki, E. Segawa, A. Suzuki and K. Suzuki, 
Scattering and inverse scattering for nonlinear quantum walks, 
Discrete Contin. Dyn. Syst. 38 (2018), 
3687--3703. 



\bibitem{MorawetzStrauss1973}
C. Morawetz and W.A. Strauss, 
On a nonlinear scattering operator, 
Comm. Pure Appl. Math. 26 (1973) 47--54.



\bibitem{NakamuraOzawa1998}
M. Nakamura and T. Ozawa,
Nonlinear Schr\"odinger equations 
in the Sobolev space of critical order,
J. Funct. Anal. 155 (1998) 364--380. 



\bibitem{NakamuraOzawa1999}
M. Nakamura and T. Ozawa, 
Global solutions in the critical Sobolev space 
for the wave equations with nonlinearity 
of exponential growth, 
Math. Z. 231 (1999), 479--487.



\bibitem{NakamuraOzawa2001}
M. Nakamura and T. Ozawa, 
The Cauchy problem for nonlinear
Klein-Gordon equations in the Sobolev spaces, 
Publ. RIMS, Kyoto Univ. 37 (2001), 
255--293.



\bibitem{OgawaOzawa1991}
T. Ogawa and T. Ozawa,
Trudinger type inequalities and uniqueness of weak solutions 
for the nonlinear Schr\"odinger mixed problem,
J. Math. Anal. Appl. 155 (1991),
531--540.



\bibitem{Pecher1976}
H. Pecher,
$L^p$-Absch\"atzungen und klassische 
L\"osungen f\"ur nichtlineare Wellengleichungen. I,
Math. Z. 150 (1976),
159--183.



\bibitem{Sasaki2007}
H. Sasaki, 
The inverse scattering problem for Schr\"odinger and 
Klein-Gordon equations with a nonlocal nonlinearity, 
Nonlinear Anal. 66 (2007) 1770--1781.




\bibitem{Sasaki2008}
H. Sasaki, 
Inverse scattering for the nonlinear Schr\"{o}dinger equation with the Yukawa potential, 
Comm. Partial Differential Equations 33 (2008), 
1175--1197. 



\bibitem{Sasaki2012}
H. Sasaki, 
Inverse scattering problems for the Hartree equation 
whose interaction potential decays rapidly, 
J. Differential Equations 252 (2012) 2004--2023. 



\bibitem{Sasaki2024}
H. Sasaki,
Remark on the scattering operator 
for the two-dimensional nonlinear Klein-Gordon equation 
with exponential nonlinearity,
Kyushu J. Math. 78 (2024), 91--117.



\bibitem{Sasaki2024-2}
H. Sasaki, 
On inverse scattering 
for the two-dimensional nonlinear Schr\"{o}dinger equation, 
J. Differential Equations 401 (2024), 308--333.



\bibitem{SasakiSuzuki2011}
H. Sasaki and A. Suzuki, 
An inverse scattering problem for the Klein-Gordon equation with a classical source in quantum field theory, 
Hokkaido Math. J. 40 (2011), 
149--186. 



\bibitem{SasakiWatanabe2005}
H. Sasaki and M. Watanabe, 
Uniqueness on identification of cubic convolution nonlinearity, 
J. Math. Anal. Appl. 309 (2005) 294--306.



\bibitem{Strauss1974}
W.A. Strauss, Nonlinear scattering theory, 
in: J.A. Lavita, J.-P. Marchand (Eds.), 
Scattering Theory in Math. Physics, 
Reidel, Dordrecht, 1974, pp. 53--78.



\bibitem{Wang1999}
B. Wang, 
Scattering of solutions for critical and 
subcritical nonlinear Klein-Gordon equations 
in $H^s$, 
Discrete Contin. Dynam. Systems 5 (1999), 
753--763.



\bibitem{Watanabe2002}
M. Watanabe,
Reconstruction of the Hartree-type nonlinearity,
Inverse Problems 18 (2002), 
1477--1481.



\bibitem{Weder1997}
R. Weder,
Inverse scattering for the nonlinear 
Schr\"{o}dinger equation, 
Comm. Partial Differential Equations 22 (1997), 
2089--2103.



\bibitem{Weder2000}
R. Weder, 
Inverse scattering on the line for the 
nonlinear Klein-Gordon equation with a potential, 
J. Math. Anal. Appl. 252 (2000), 
102--123.



\bibitem{Weder2000-2}
R. Weder,
$L^p$-$L^{p^\prime}$
estimates for the Schr\"{o}dinger equation 
on the line and inverse scattering 
for the nonlinear Schr\"{o}dinger equation 
with a potential, 
J. Funct. Anal. 170 (2000), 
37--68.



\bibitem{Weder2001}
R. Weder,
Inverse scattering for the non-linear 
Schr\"{o}dinger equation: 
reconstruction of the potential and the non-linearity, 
Math. Methods Appl. Sci. 24 (2001), 
245--254.



\bibitem{Weder2001-2}
R. Weder,
Inverse scattering for the nonlinear 
Schr\"{o}dinger equation. II. 
Reconstruction of the potential and the nonlinearity 
in the multidimensional case, 
Proc. Amer. Math. Soc. 129 (2001),
3637--3645.



\bibitem{Weder2002}
R. Weder, 
Multidimensional inverse scattering 
for the nonlinear Klein-Gordon equation with a potential,
J. Differential Equations 184 (2002) 62--77.




\bibitem{Weder2005-1}
R. Weder,
Scattering for the forced non-linear 
Schr\"{o}dinger equation with a potential on the half-line, 
Math. Methods Appl. Sci. 28 (2005), 
1219--1236.



\bibitem{Weder2005-2}
R. Weder,
The forced non-linear Schr\"{o}dinger equation with 
a potential on the half-line, 
Math. Methods Appl. Sci. 28 (2005), 
1237--1255.
\end{thebibliography}
\end{document}